\documentclass[12pt]{ip-journal}
\usepackage{amsmath}
\usepackage{amsthm}
\usepackage{amssymb}
\usepackage{amsmath}
\usepackage{amsthm}
\usepackage{amssymb}
\usepackage{fancyhdr}
\usepackage{enumerate}
\usepackage{amscd}
\usepackage[all]{xy}
\usepackage{graphicx}
\usepackage{color}
\usepackage{hyperref}
\hypersetup{
    colorlinks,
    citecolor=black,
    filecolor=black,
    linkcolor=black,
    urlcolor=black
}
\usepackage{tikz}
\usepackage[all]{xy}
\usepackage{graphicx}
\usetikzlibrary{backgrounds}

\newtheorem{theorem}{Theorem}[section]
\newtheorem{corollary}[theorem]{Corollary}
\newtheorem{lemma}[theorem]{Lemma}

\newtheorem{claim}[theorem]{Claim}

\theoremstyle{definition}
\newtheorem{definition}[theorem]{Definition}

\theoremstyle{remark}
\newtheorem{remark}[theorem]{Remark}

\numberwithin{equation}{section}

\newcommand{\N}{\mathbb N}
\newcommand{\C}{\mathbb C}
\newcommand{\Sp}{\mathbb S}

\newcommand{\B}{\mathcal B}

{{\sc Proof of Theorem~\ref{cannon}.}}%
{{\qed} \\}

{{\sc Proof of Theorem~\ref{degree}.}}%
{{\qed} \\}

{{\sc Proof of Proposition~\ref{same}.}}%
{{\qed} \\}

{{\sc Proof of (C1).}}%
{{\qed} \\}

{{\sc Proof of (C2).}}%
{{\qed} \\}

{{\sc Proof of (C3).}}%
{{\qed} \\}
\allowdisplaybreaks

\date{\today}
\title[Existence of harmonic maps into CAT(1) spaces]
{Existence of harmonic maps into CAT(1) spaces}

\thanks{This work began as part of the workshop ``Women in Geometry" (15w5135) at the Banff International Research Station in November of 2015. We are grateful to BIRS for the opportunity to attend and for the excellent working environment. 
CB, CM were supported in part by NSF grants DMS-1308420 and DMS-1406332 respectively, and LH was supported by NSF grants DMS-1308837 and DMS-1452477. AF was supported in part by an NSERC Discovery Grant. PS was supported in part by an NSERC PGS D scholarship and a UBC Four Year Doctoral Fellowship. YZ was supported in part by an AWM-NSF Travel Grant. This material is also based upon work supported by NSF DMS-1440140 while CB and AF were in residence at the Mathematical Sciences Research Institute in Berkeley, California, during the Spring 2016 semester.}
\author[Breiner]{Christine Breiner}
\address{Department of Mathematics \\
                 Fordham University \\Bronx, NY  10458\\}
\email{cbreiner@fordham.edu}

\author[Fraser]{Ailana Fraser}
\address{Department of Mathematics \\
                 University of British Columbia \\
                 Vancouver, BC V6T 1Z2\\}
\email{afraser@math.ubc.ca}

\author[Huang]{Lan-Hsuan Huang}
\address{Department of Mathematics\\ University of Connecticut\\ Storrs, CT 06269\\}
\email{lan-hsuan.huang@uconn.edu}

\author[Mese]{Chikako Mese}
\address{Johns Hopkins University\\
Department of Mathematics\\
3400 N. Charles Street\\
Baltimore, MD  21218\\
}\email{cmese@math.jhu.edu}

\author[Sargent]{Pam Sargent}
\address{Department of Mathematics \\
                 University of British Columbia \\
                 Vancouver, BC V6T 1Z2\\}
\email{psargent@math.ubc.ca}

\author[Zhang]{Yingying Zhang}
\address{Johns Hopkins University\\
Department of Mathematics\\
3400 N. Charles Street\\
Baltimore, MD  21218\\}
\email{yzhang@math.jhu.edu}

\begin{document}

\maketitle

\section{Introduction}

In many existence theorems for harmonic maps, the key assumption is the  non-positivity of the curvature of the target space.  The prototype is the  celebrated work of Eells and Sampson \cite{eells-sampson}  and Al'ber \cite{alber1}, \cite{alber2}  where the assumption of  the non-positive sectional curvature of the target Riemannian manifold plays an essential  role.  The Eells-Sampson  existence theorem  has been extended to the equivariant case  by Diederich-Ohsawa \cite{diederich-ohsawa},  Donaldson \cite{donaldson}, Corlette \cite{corlette}, Jost-Yau \cite{jost-yau} and Labourie \cite{labourie}.  Again, all these works assume non-positive sectional curvature on the target.  
  For smooth Riemannian manifold domains and NPC targets (i.e.~complete metric spaces with non-positive curvature in the sense of Alexandrov), existence theorems were obtained by Gromov-Schoen \cite{gromov-schoen}  and Korevaar-Schoen \cite{korevaar-schoen1}, \cite{korevaar-schoen2}.   The generalization to the case when the domain is a metric measure space has been discussed by Jost (\cite{jostgreenbook} and the references therein) and separately by Sturm \cite{sturm}.

When the curvature of the target space is not assumed to be non-positive, the existence problem for harmonic maps becomes more complicated, and in many ways, more interesting.  Although the general problem is not well understood, a breakthrough was achieved in the case of two-dimensional domains by Sacks and Uhlenbeck \cite{sacks-uhlenbeck}.   Indeed, they  discovered a ``bubbling phenomena" for harmonic maps;    more specifically, they prove the following dichotomy: given a finite energy map from a Riemann surface into a compact Riemannian manifold,  either there exists a  harmonic map homotopic to the given map  or there exists a branched minimal immersion of the 2-sphere.  We also mention the  related works of  Lemaire \cite{lemaire}, Sacks-Uhlenbeck \cite{sacks-uhlenbeck2}, and Schoen-Yau \cite{schoen-yau}. 

The goal of this paper is prove an analogous  result when the target space is a compact locally CAT(1) space, i.e.~a compact complete metric space of curvature bounded above by 1 in the sense of Alexandrov.

\begin{theorem} \label{main}
Let $\Sigma$ be a compact Riemann surface, $X$ a compact locally CAT(1) space and $\varphi \in C^0 \cap W^{1,2}(\Sigma,X)$. Then either there exists a harmonic map $u:\Sigma \rightarrow X$ homotopic to $\varphi$ or a conformal harmonic map $v:\mathbb S^2 \rightarrow X$.
\end{theorem}

Sacks and Uhlenbeck  used the perturbed energy method in the proof of Theorem~\ref{main} for Riemannian manifolds.  In doing so, they rely heavily on a priori estimates procured from the Euler-Lagrange equation of the perturbed energy functional.  
One of the difficulties in working in the singular setting is that,  because of the lack of local coordinates, one  does not have a P.D.E.  derived from a variational principle  (e.g.~harmonic map equation). In order to prove results in the singular setting, we cannot rely on P.D.E. methods. 
To this end, we use a 2-dimensional generalization of the Birkhoff curve shortening method \cite{birkhoff1}, \cite{birkoff2}.  This  local replacement process can be thought of as a discrete gradient flow.   This idea was  used by Jost \cite{jost} to give an alternative proof of the Sacks-Uhlenbeck theorem in the smooth setting.  More recently,  in studying width and proving finite time extinction of the Ricci flow, Colding-Minicozzi \cite{CM} further developed the local replacement argument and proved a new convexity result for harmonic maps and continuity of harmonic replacement. However, even these arguments rely on the harmonic map equation and  hence do not  translate to our case.   The main accomplishment of our method is to eliminate the need for a P.D.E. by using the local convexity properties of the target CAT(1) space.  (The  necessary convexity properties of a CAT(1) space are given in Appendices \ref{section:quadrilateral-estimates} \& \ref{energy-convexity}.)

For clarity, we provide a brief outline of the harmonic replacement construction. Given $\varphi:\Sigma \to X$, we set $\varphi = u_0^0$ and inductively construct a sequence of energy decreasing maps $u_n^l$ where $n \in \mathbb N \cup \{0\}$, $l \in \{0, \dots, \Lambda\}$, and $\Lambda$ depends on the geometry of $\Sigma$. The sequence is constructed inductively as follows. Given the map $u_n^0$, we determine the largest radius, $r_n$, in the domain on which we can apply the existence and regularity of Dirichlet solutions (see Lemma \ref{dir}) for this map. Given a suitable cover of $\Sigma$ by balls of this radius, we consider $\Lambda$ subsets of this cover such that every subset consists of non-intersecting balls. 
The maps $u_n^l:\Sigma \to X$, $l \in \{1, \dots, \Lambda\}$ are determined by replacing $u_n^{l-1}$ by its Dirichlet solution on balls in the $l$-th subset of the covering and leaving the remainder of the map unchanged. We then set $u_{n+1}^0:= u_n^\Lambda$ to continue by induction. There are 
now two possibilities, depending on $\liminf r_n=r$. If $r>0$, we demonstrate that the sequence we constructed is equicontinuous and has a unique limit that is necessarily homotopic to 
$\varphi$. Compactness for minimizers (Lemma \ref{luckhaus}) then implies that the limit map is harmonic. If $r=0$, then bubbling occurs. That is, after an appropriate rescaling of the original 
sequence, the new sequence is an equicontinuous family of harmonic maps from domains exhausting $\mathbb C$. As in the previous case, this sequence converges on compact sets to a limit harmonic map from $\mathbb C$ to $X$. We extend this map to $\mathbb S^2$ by a removable singularity theorem developed in section \ref{section:removable-singularity}.

 We now give an outline of the paper. In section \ref{section:preliminary}, we introduce some notation and provide the results that are necessary in order to perform harmonic replacement and obtain a harmonic limit map. In particular, we state the existence and regularity results for Dirichlet solutions and  prove compactness of energy minimizing maps into a CAT(1) space. 
 In section \ref{section:removable-singularity}, we prove our removable singularity theorem. Namely, in Theorem \ref{singularity} we prove that any conformal harmonic map from a punctured surface into a CAT(1) space extends as a locally Lipschitz harmonic map on the surface. This theorem extends to CAT(1) spaces the removable singularity theorem of Sacks-Uhlenbeck \cite{sacks-uhlenbeck} for a finite energy harmonic map into a Riemannian manifold, provided the map is conformal. The proof relies on two key ideas. First, for harmonic maps $u_0$ and $u_1$ into a CAT(1) space, while $d^2(u_0,u_1)$ is not subharmonic, 
 a more complicated weak differential inequality holds if the maps are into a sufficiently small ball (Theorem \ref{theorem:subharmonicity} in Appendix \ref{energy-convexity}, \cite{serbinowski}). Using this inequality, we prove a local removable singularity theorem for harmonic maps into a small ball. The second key idea, Theorem \ref{monotonicity}, is a monotonicity of the area in extrinsic balls in the target space, for conformal harmonic maps from a surface to a CAT(1) space. This theorem extends the classical monotonicity of area for minimal surfaces in Riemannian manifolds to metric space targets. The proof relies on the fact that the distance function from a point in a CAT(1) space is almost convex on a small ball. In application, the monotonicity is used to show that a conformal harmonic map defined on $\Sigma \backslash \{p\}$ is continuous across $p$. Then the local removable singularity theorem can be applied at some small scale. 
 Section \ref{section:harmonic-replacement} contains the harmonic replacement construction outlined above and the proof of the main theorem, Theorem \ref{main}. Finally, in Appendix \ref{section:quadrilateral-estimates} we give complete proofs of several difficult estimates for quadrilaterals in a CAT(1) space. The estimates are stated in the unpublished thesis \cite{serbinowski} without proof. We apply these estimates in Appendix \ref{energy-convexity} to give complete proofs of some energy convexity, existence, uniqueness, and subharmonicity results (also stated in \cite{serbinowski}) that are used throughout this paper.

\section{Preliminary results} \label{section:preliminary}

Throughout the paper we let $(\Omega, g)$ denote a Lipschitz Riemannian domain and $(X, d)$ a locally CAT(1) space. We refer the reader to Section 2.2 of \cite{BFHMSZ} for some background on CAT(1) spaces. We define the Sobolev space  $ W^{1,2}(\Omega,X) \subset L^2(\Omega,X)$ of finite energy maps. In particular, if $u \in W^{1,2}(\Omega,X)$, one can define its energy density  $|\nabla u|^2 \in L^1(\Omega)$  and the total energy 
\[
	{^d}E^u [\Omega] =\int_{{\Omega}}|\nabla u|^2 d\mu_g.
\]
We often suppress the superscript $d$ when the context is clear. We refer the reader to \cite{korevaar-schoen1} for further details and background. We denote a geodesic ball in $\Omega$ of radius $r$ centered at $p \in \Omega$ by $B_r(p)$ and a geodesic ball in $X$ of radius $\rho$ centered at $P \in X$ by $\B_\rho(P)$.  
Furthermore, given $h \in W^{1,2}(\Omega,X)$, we define
\[
W^{1,2}_h(\Omega,X)=\{ f \in W^{1,2}(\Omega,X): Tr(h)=Tr(f)\},
\]
where $Tr(u) \in L^2(\partial \Omega, X)$ denotes the trace map of $u \in W^{1,2}(\Omega,X)$ (see \cite{korevaar-schoen1} Section 1.12).

\begin{definition}
We say that a map $u:\Omega \rightarrow X$ is \emph{harmonic} if it is locally energy minimizing with locally finite energy; precisely, for every $p \in \Omega$, there exist $r >0$, $\rho>0$ and $P \in X$ such that $h=u\big|_{B_{r}(p)}$ has finite energy and minimizes energy among all maps in $W^{1,2}_h(B_r(p),\overline{\B_\rho(P)})$.  
\end{definition}

The following results will be used in the proof of the main theorem, Theorem \ref{main}.

\begin{lemma}[Existence, Uniqueness and Regularity of the Dirichlet solution] \label{dir}
For any finite energy map $h: \Omega \rightarrow \overline{\B_{\rho}(P)} \subset X$, where $\rho \in (0,\frac{\pi}{4})$, the Dirichlet solution exists. That is, there exists a unique element $^{Dir}h \in W^{1,2}_h(\Omega, \overline{\mathcal{B}_{\rho}(P)})$ that minimizes energy among all maps in $W^{1,2}_h(\Omega,\overline{\mathcal{B}_{\rho}(P)})$. 
Moreover, if $^{Dir}h(\partial \Omega) \subset \overline{\B_{\sigma}(P)}$ for some $\sigma \in (0,\rho)$, then $\overline{^{Dir}h({\Omega}) }\subset \overline{\B_{\sigma}(P)}$.  Finally, the solution $^{Dir}h$  is locally Lipschitz continuous with Lipschitz constant depending only on the total energy of the map and the metric on the domain. 
\end{lemma}
For further details see Lemma \ref{existence} in Appendix \ref{energy-convexity}, \cite{serbinowski}, and \cite{BFHMSZ}.

\begin{lemma}[Compactness for minimizers into CAT(1) space] \label{luckhaus}
Let $(X,d)$ be a CAT(1) space and $B_r\subset \Omega$ a geodesic (and topological) ball of radius $r>0$ where $(\Omega,g)$ is a Riemannian manifold. Let $u_i:B_r  \to X$ be a sequence of energy minimizers with $E^{u_i}[B_r] \leq \Lambda$ for some $\Lambda>0$. 

Suppose that $u_i$ converges uniformly to $u$ on $B_r$ and that
there exists $P \in X$ such that $u(B_{r}) \subset \mathcal B_{\rho/2}(P)$ where $\rho$ is as in Lemma \ref{dir}. Then $u$ is energy minimizing on $B_{r/2}$.
\end{lemma}


\begin{proof}We will follow the ideas of the proof of Theorem 3.11 \cite{korevaar-schoen2}. Rather than prove the bridge principle for CAT(1) spaces, we will modify the argument and appeal directly to the bridge principle for NPC spaces (see Lemma 3.12 \cite{korevaar-schoen2}).  

Since $u_i \to u$ uniformly and $u(B_r) \subset \B_{\rho/2}(P)$, there exists $I$ large such that for all $i \geq I$, $u_i(B_r) \subset \B_\rho(P)$. By Lemma \ref{dir}, there exists $c>0$ depending only on $\Lambda$ and $g$ such that for all $i \geq I$, $u_i|_{B_{3r/4}}$ is Lipschitz with Lipschitz constant $c$. It follows that for $t>0$ small, there exists $C>0$ depending on $c$ and the dimension of $\Omega$ such that
\begin{equation}\label{kssmall}
E^{u_i}[B_{r/2}\backslash B_{r/2-t}] \leq Ct.
\end{equation}For $\varepsilon>0$, increase $I$ if necessary so that
for all $i \geq I$ and all $x \in  B_{3r/4}$, 
\begin{equation}\label{pointwiseLuck}
 d^2(u_i(x),u(x))<  \varepsilon.
\end{equation}

For notational ease, let $U_t:=B_{r/2-t}$. Let $w_t:U_t \to X$ denote the energy minimizer $w_t:={} ^{Dir}u|_{U_t} \in W^{1,2}_u(U_t,X)$, with existence guaranteed by Lemma \ref{dir}. Following the argument in the proof of Theorem 3.11 \cite{korevaar-schoen2}, \eqref{kssmall} and the lower semi-continuity of the energy imply that $\lim_{t\to 0} E^{w_t}[U_t] = E^{w_0}[B_{r/2}]$. 
Observe that by the lower semi-continuity of energy, Theorem 1.6.1 \cite{korevaar-schoen1},
\[
{^d E}^u[B_{r/2}] \leq \liminf_{i \to \infty} {^d E}^{u_i}[B_{r/2}].
\]Thus, it will be enough to show that 
\[
 \limsup_{i\to \infty} {^d E}^{u_i}[B_{r/2}] \leq {^d E}^{w_0}[B_{r/2}].
\]

Let $v_t:B_{r/2} \to X$ be the map such that $v_t|_{U_t}=w_t$ and $v_t|_{B_{r/2} \backslash U_t}=u$. Given $\delta>0$, choose $t>0$ sufficiently small so that
\begin{equation}\label{delta_t_choice}
^{d}E^{v_t}[B_{r/2}] < {}^dE^{w_0}[B_{r/2}] + \delta.
\end{equation}

Since $v_t$ is not a  competitor for $u_i$ (i.e. $v_t|_{\partial B_{r/2}}$ is not necessarily equal to $u_i|_{\partial B_{r/2}}$), for each $i$ we want to bridge from $v_t$ to $u_i$ for values near $\partial B_{r/2}$. 
Since we want to exploit a bridging lemma into NPC spaces, rather than bridge between $v_t$ and $u_i$, we will bridge between their lifted maps in the cone $\mathcal C(X)$.

Let $\mathcal C(X):=(X \times [0,\infty)/X \times \{0\},D  )$ where
\[
D^2([P,x],[Q,y])=x^2+y^2-2xy\cos \min (d(P,Q),\pi).
\] Then $\mathcal C(X)$ is an NPC space and we can identify $X$ with $X \times \{1\} \subset \mathcal C(X)$. 
For any map $f:B_r \to X$, we let $\overline f:B_r \to X \times \{1\}$ such that $\overline f(x) = [f(x),1]$. Note that for $f\in W^{1,2}(B_r, \B_\rho(Q))$, since
\begin{equation*}\label{projectiondistance}
\lim_{P \to Q}\frac{D^2([P,1],[Q,1])}{d^2(P,Q)} = \lim_{P \to Q}\frac{2(1- \cos (d(P,Q)))}{d^2(P,Q)}=1,
\end{equation*}it follows that ${}^DE^{\overline f}[\Omega] = {}^dE^f[\Omega]$ for $\Omega \subset B_r$.

For each $i\geq I$, and a fixed $s,\rho>0$ to be chosen later, define the map
\[
 v_i: \partial U_s \times [0,\rho] \to \mathcal C(X)
\]such that
\[
 v_i(x,z):=\left(1-\frac z\rho \right)\overline v_t(x) + \frac z{\rho} \overline u_i(x).
\] 

The map $v_i$ is a bridge between $\overline v_t|_{\partial U_{s}}$ and $\overline u_i|_{\partial U_{s}}$ in the NPC space $\mathcal C(X)$. 
That is, we are interpolating along geodesics connecting $\overline v_t(x), \overline u_i(x)$ in the NPC space 
$\mathcal C(X)$ and not  along geodesics in $X$. 
By \cite{korevaar-schoen2} (Lemma 3.12) and the equivalence of the energies for a map $f$ and its lift $\overline f$, 
\begin{align*}
{^D E}^{ v_i} [\partial U_{s} \times [0,\rho]]&\leq \frac \rho 2 \left({^D E}^{  \overline v_t} [\partial U_{s} ]+ {^D E}^{ \overline u_i} [\partial U_{s} ] \right) + \frac 1\rho \int_{\partial U_{s} } D^2([v_t,1],[u_i,1]) d\sigma \\
& =  \frac \rho 2 \left({^d E}^{   v_t} [\partial U_{s} ]+ {^d E}^{  u_i} [\partial U_{s} ] \right) + \frac 1\rho \int_{\partial U_{s} } D^2([v_t,1],[u_i,1]) d\sigma .
\end{align*}

By \eqref{kssmall}, and since $v_t=u$ on $B_{r/2} \backslash U_t$, for $s \in [2t/3,3t/4]$ the average values of the tangential energies of $v_t$ and $u_i$ on $\partial U_s$ are bounded above by $Ct/(3t/4-2t/3)=12C$. Moreover, since 
 $u_i(B_{r/2}),v_t(B_{r/2}) \subset \mathcal B_{\rho}(P)$, \eqref{pointwiseLuck} implies that for all $x \in B_{r/2}\backslash U_t$, 
\begin{equation}\label{Distance}
 D^2(\overline u_i(x), \overline v_t(x))= 2(1- \cos d(u_i(x), v_t(x))) \leq d^2(u_i(x), v_t(x))< \varepsilon .
\end{equation}
Thus, there exists $C'>0$ depending only on $g$ such that for every $s \in [2t/3, 3t/4]$,
\[
\int_{\partial U_s} D^2([v_t,1],[u_i,1]) d\sigma < C'\varepsilon.
\]Note that for each $\varepsilon>0$, the bound above depends on $I$ but not on $t$. 
Now, we first choose an $s \in (2t/3, 3t/4)$ such that ${^d E}^{   v_t} [\partial U_{s} ]+ {^d E}^{  u_i} [\partial U_{s} ] \leq 24C$. Next, pick $0<\mu \ll 1$ such that $[s, s+ \mu t] \subset [2t/3, 3t/4]$ and $12C\mu t < \delta/2$.  For this $t,\mu$, decrease $\varepsilon$ if necessary (by increasing $I$) such that 
\begin{align*}
^D E^{v_i}[\partial U_s \times [0, \mu t]] &= \frac {\mu t} 2 \left({^d E}^{   v_t} [\partial U_{s} ]+ {^d E}^{  u_i} [\partial U_{s} ] \right) + \frac 1{\mu t} \int_{\partial U_{s} } D^2([v_t,1],[u_i,1]) d\sigma \\
&< 24C \mu t /2+ C'\varepsilon/(\mu t)\\
& < \delta.
\end{align*}
Now, define $\tilde v_i:B_{r/2} \to \mathcal C(X)$ such that on $U_s$, $\tilde v_i$ is the conformally dilated map of $\overline v_t$ so that $\tilde v_i|_{\partial U_{s+\mu t}} = \overline v_t|_{\partial U_s}$. On $U_s \backslash U_{s+\mu t}$, let $\tilde v_i $ be the bridging map $v_i$, reparametrized in the second factor from $[0, \mu t]$ to $[s, s + \mu t]$. Finally, on $B_{r/2} \backslash U_s$, let $\tilde v_i= \overline u_i$. Then, for all $i \geq I$, 
\begin{equation}\label{tilde_est}
^DE^{\tilde v_i}[B_{r/2}] \leq {}^dE^{v_t}[B_{r/2}] + \delta + {}^dE^{u_i}[B_{r/2} \backslash U_s].
\end{equation}

While the map $\tilde v_i$ agrees with $\overline u_i$ on $\partial B_{r/2}$, it is not a competitor for $u_i$ into $X$ since $\tilde v_i$ maps into $\mathcal C(X)$. However, by defining 
$\underline v_i:B_{r/2} \to X$ such that $\tilde v_i(x) = [\underline v_i(x),h(x)]$, $\underline v_i$ is a competitor. Note that for all $x\in \partial U_s$, \eqref{Distance} implies that $h(x) \geq 1-\sqrt{\varepsilon}$. Therefore, on the bridging strip we may estimate the change in energy under
the projection map by first observing the pointwise bound 
\begin{align*}
D^2(\tilde v_i(x),\tilde v_i(y))&= D^2([\underline v_i(x),h(x)],[\underline v_i(y),h(y)])
\\&= h(x)^2 + h(y)^2 - 2h(x)h(y) \cos (d(\underline v_i(x),\underline v_i(y)))\notag\\
&= (h(x)-h(y))^2 + 2h(x)h(y) (1-\cos (d(\underline v_i(x),\underline v_i(y)))) \notag\\
& \geq 2(1-\sqrt{\varepsilon})^2(1-\cos(d(\underline v_i(x),\underline v_i(y))))\label{distancelowerbnd}\\
&= (1-\sqrt{\varepsilon})^2 D^2([\underline v_i(x),1],[\underline v_i(y),1]).\notag
\end{align*}
Therefore,
\begin{equation}\label{underline_est}
{^d E}^{\underline v_i}[B_{r/2}] = {}{^D E}^{[\underline v_i,1]}[B_{r/2}] \leq \left(1-\sqrt{\varepsilon}\right)^{-2}\, {}{^D E}^{\tilde v_i}[B_{r/2}] .
\end{equation}
Since $\underline v_i$ is a competitor for $u_i$ on $B_{r/2}$, \eqref{underline_est},  \eqref{tilde_est},  \eqref{delta_t_choice}, and \eqref{kssmall}  imply that
\begin{align*}
 {^d E}^{u_i}[B_{r/2}] &\leq  \left(1-\sqrt{\varepsilon}\right)^{-2}\, {^D E}^{\tilde v_i}[B_{r/2}]\leq   \left(1-\sqrt{\varepsilon}\right)^{-2}\,\left({^d E}^{w_0}[B_{r/2}]+2\delta+Ct\right)
\end{align*}
Since for any $\varepsilon,\delta>0$, by choosing $t>0$ sufficiently small and $I\in \mathbb N$ large enough, the previous estimate holds for all $i \geq I$, the inequality
\[
 \limsup_{i\to \infty} {^d E}^{u_i}[B_{r/2}] \leq {^d E}^{w_0}[B_{r/2}]
\]then implies the result.

\end{proof}

\section{Monotonicity and removable singularity theorem} \label{section:removable-singularity}

We first show the removable singularity theorem for harmonic maps into small balls.

\begin{theorem}\label{theorem:minimizer}
Let $u: B_r(p)\setminus \{ p \} \to  \mathcal{B}_\rho(P)\subset  X$ be a finite energy harmonic map, where $\rho$ is as in Lemma \ref{dir}. Then $u$ can be extended on $B_r(p)$ as the unique energy minimizer among all maps in $W_u^{1,2}(B_r(p),  \mathcal{B}_\rho(P))$. 
\end{theorem}
\begin{proof}
 Let $v\in W_{u}^{1,2}(B_r(p),  \mathcal{B}_\rho(P))$ minimize the energy. It suffices to  show that $u=v$ on $B_r(p)\setminus \{ p \}$.  Since $u$ is harmonic, there exists a locally finite countable open cover $\{ U_i\}$ of $B_r(p)\setminus \{p\}$, and $\rho_i>0, P_i\in \mathcal{B}_\rho(P)$ such that $u|_{U_i}$ minimizes energy among all maps in $W^{1,2}_{u}(U_i, \mathcal{B}_{\rho_i}(P_i))$. 
 Let 
\[
	F = \sqrt{\frac{1-\cos d}{\cos R^{u} \cos R^{v}}}
\]
where $d(x)  = d(u(x), v(x))$ and $R^u = d(u, P), R^v = d(v, P)$. By Theorem \ref{theorem:subharmonicity}, 
\[
	\textup{div} (\cos R^{u} \cos R^{v} \nabla F)\ge 0
\] 
holds weakly on each $U_i$. Therefore, for a partition of unity $\{ \varphi_i \}$  subordinate to the cover $\{ U_i \}$ and for any test function  $\eta \in C^{\infty}_c(B_r(p)\setminus \{ p \})$,
 \begin{align} \label{equation:subharmonicity}
 	&- \int_{B_r(p)\setminus \{ p \}}  \nabla \eta \cdot  (\cos R^{u} \cos R^{v} \nabla F) \, d\mu_g = - \sum_i\int_{U_i}  \nabla (\varphi_i \eta) \cdot  (\cos R^{u} \cos R^{v} \nabla F) \, d\mu_g \ge 0,
 \end{align}
 where we use $\sum_i \varphi_i = 1$ and $\sum_i \nabla \varphi_i = 0$.  
 
Using polar coordinates in $B_r(p)$ centered at $p$, for $0<\epsilon\ll 1$, we define
 \[
 	\phi_{\epsilon } =\left\{ \begin{array} {ll} 0 & r \le \epsilon^2  \\ \frac{\log r - \log \epsilon^2 }{- \log \epsilon} & \epsilon^2 \le  r \le \epsilon \\ 1 & \epsilon \le r  \end{array} \right..
 \]
Note that 
 \[
 	\int_{B_r(p)} | \nabla \phi_{\epsilon} |^2\, d\mu_g = \frac{2\pi}{(\log \epsilon)^2} \int_{\epsilon^2}^{\epsilon} r^{n-1} \, dr + o (\epsilon)  \to 0 \quad \mbox{ as } \epsilon \to 0.
 \]
 Therefore, for $\eta \in C^\infty_c(B_r(p))$,
 \begin{align*}
 	&-\int_{B_r(p)} \phi_{\epsilon}   \nabla \eta \cdot (\cos R^{u} \cos R^{v} \nabla F)  \, d\mu_g \\
	& =- \int_{B_r(p)} \nabla (\eta \phi_{\epsilon}) \cdot  (\cos R^{u} \cos R^{v} \nabla F) \, d\mu_g + \int_{B_r(p)} \eta \nabla \phi_{\epsilon} \cdot  (\cos R^{u} \cos R^{v} \nabla F) \, d\mu_g\\
	& \ge    \int_{B_r(p)\setminus\{p\}} \eta \nabla \phi_{\epsilon} \cdot   (\cos R^{u} \cos R^{v} \nabla F)\, d\mu_g \qquad (\textup{by } \eqref{equation:subharmonicity})\\
	& \ge - \left(\int_{B_r(p)\setminus \{p\}} |\nabla \phi_{\epsilon}|^2 \, d\mu_g\right)^{\frac{1}{2}} \left( \int_{{B_r(p)}\setminus \{ p \}} \eta^2 |\cos R^{u} \cos R^{v} \nabla F|^2 \, d\mu_g\right)^\frac{1}{2}  \qquad \mbox{ (by H\"{o}lder's inequality)}.
 \end{align*}
The last line converges to zero as $\epsilon\to 0$ because $d, R^{u}, R^{v}$ are bounded by the compactness of $\overline{\mathcal{B}_\rho(P)}$ and $\int_{{B_r(p)}\setminus \{ p \}} |\nabla F|^2\, d\mu_g$ is bounded by energy convexity. We conclude that 
\[
	-\int_{B_r(p)}  \nabla \eta \cdot  (\cos R^{u} \cos R^{v} \nabla F)\, d\mu_g =-\lim_{\epsilon \to 0} \int_{B_r(p)}  \phi_{\epsilon}  \nabla \eta \cdot  (\cos R^{u} \cos R^{v} \nabla F) \, d\mu_g \ge 0,
\]
and hence $\textup{div} (\cos R^{u} \cos R^{v} \nabla F) \ge 0$ holds weakly on ${B_r(p)}$.  

Since $d(u(x), v(x)) = 0$ on $\partial {B_r(p)}$, by the maximum principle  $d(u(x), v(x)) \equiv 0$ in ${B_r(p)}$. This implies that $u\equiv v$ is the unique energy minimizer.

\end{proof}

\begin{remark}
Note that Theorem \ref{theorem:minimizer} implies that if $u:\Omega \to \B_\rho(P)$ is harmonic, then $u$ is energy minimizing.
\end{remark}

From this point on we assume our domain  is of dimension 2. 
Recall the construction in \cite{korevaar-schoen1} and \cite{BFHMSZ} of a continuous, symmetric, bilinear, non-negative tensorial operator 
\begin{equation} \label{pi}
\pi^u: \Gamma(T\Omega) \times \Gamma(T\Omega) \rightarrow L^1(\Omega)
\end{equation}
associated  with a $W^{1,2}$-map $u: \Omega \rightarrow X$ where $\Gamma(T\Omega)$ is the space of Lipschitz vector fields on $\Omega$ defined by 
\[
\pi^u(Z,W) := \frac{1}{4} |u^*(Z+W)|^2-\frac{1}{4}|u^*(Z-W)|^2
\]
where $|u^*(Z)|^2$ is the directional energy density function (cf. \cite[Section 1.8]{korevaar-schoen1}).
This generalizes the notion of the pullback metric for maps into a Riemannian manifold, and hence we shall refer to $\pi=\pi^u$ also as the pullback metric for $u$. 
\begin{definition}
If $\Sigma$ is a Riemann surface, then $u \in W^{1,2}(\Sigma,X)$ is {\em (weakly) conformal} if
\[
    \pi \left( \frac{\partial}{\partial x_1},\frac{\partial}{\partial x_1} \right)  
                        = \pi \left( \frac{\partial}{\partial x_2},\frac{\partial}{\partial x_2} \right) 
                      \mbox{ and } \pi \left( \frac{\partial}{\partial x_1},\frac{\partial}{\partial x_2} \right) =0,
\]
where $z=x_1+ix_2$ is a local complex coordinate on $\Sigma$.  
\end{definition}

For a conformal harmonic map $u: \Sigma  \rightarrow X$ with conformal factor $\lambda=\frac{1}{2}|\nabla u|^2$, and any open sets $S \subset \Sigma$ and $\mathcal{O} \subset X$, define
\[
     A(u(S) \cap \mathcal{O}) :=\int_{u^{-1}(\mathcal{O})\cap S}  \lambda \; d\mu_g,
\]
where $d\mu_g$ is the area element of $(\Sigma,g)$.

\begin{theorem}[Monotonicity] \label{monotonicity}
There exist constants $c, \, C$ such that 
if $u: \Sigma \rightarrow X$ is a non-constant conformal harmonic map from a Riemann surface $\Sigma$ into a CAT(1) space $(X,d)$, then for any $p \in \Sigma$ and 
$0<\sigma < \sigma_0=\min \{ \rho, d(u(p),u(\partial \Sigma)) \}$, the following function is  increasing:
\[
       \sigma \mapsto \frac{e^{c\sigma^2} A(u(\Sigma) \cap \mathcal{B}_\sigma(u(p)))}{\sigma^2},
\]
and 
\[
       A(u(\Sigma) \cap \mathcal{B}_\sigma(u(p))) \geq C\sigma^2.
\]
\end{theorem}

\begin{proof}
Since $\Sigma$ is locally conformally Euclidean and the energy is conformally invariant, without loss of generality, we may assume that the domain is Euclidean. Fix $p \in \Sigma$ and let $R(x)=d(u(x), u(p))$. 
Since $u$ is continuous and locally energy minimizing, by {\cite[Proposition 1.17]{serbinowski}}, {\cite[Lemma 4.3]{BFHMSZ}} we have that the following differential inequality holds weakly on $u^{-1} (\mathcal{B}_\rho(u(p)))$:
\begin{equation} \label{inequality}
      \frac{1}{2} \Delta R^2 \geq (1-O(R^2)) |\nabla u|^2.
\end{equation}

Let $\zeta: \mathbb{R^+} \rightarrow \mathbb{R}^+$ be any smooth nonincreasing function such that $\zeta(t)=0$ for $t \geq 1$, and let $\zeta_\sigma(t)=\zeta(\frac{t}{\sigma})$. By (\ref{inequality}), 
for $\sigma < \sigma_0$ we have
\begin{align*}
     -\int_\Sigma \nabla R^2 \cdot \nabla (\zeta_\sigma ( R)) \, dx_1dx_2 
    & \geq 2 \int_\Sigma  \zeta_\sigma ( R) \, (1-O(R^2)) |\nabla u|^2 \, dx_1dx_2 \\
    & = 4 \int_\Sigma \zeta_\sigma( R) \, (1-O(R^2)) \, \lambda \, dx_1dx_2.
\end{align*}
Therefore,
\begin{align*}
     2\int_\Sigma \zeta_\sigma ( R) \, (1-O(R^2)) \, \lambda \, dx_1dx_2
     & \leq - \int_\Sigma R \, \nabla R \cdot \nabla (\zeta_\sigma (R)) \, dx_1dx_2  \\
     &= -\int_\Sigma \frac{R}{\sigma} \, \zeta' \left(\frac{R}{\sigma}\right) |\nabla R|^2 \, dx_1dx_2 \\
    &\leq -\int_\Sigma \frac{R}{\sigma} \, \zeta'\left(\frac{R}{\sigma}\right) \frac{1}{2} |\nabla u|^2 \, dx_1dx_2 \\
    &= -\int_\Sigma \frac{R}{\sigma} \zeta'\left(\frac{R}{\sigma}\right) \, \lambda \, dx_1dx_2 \\
    &= \int_\Sigma \sigma \, \frac{d}{d\sigma} (\zeta_\sigma ( R)) \; \lambda \, dx_1dx_2 \\
    &= \sigma \frac{d}{d\sigma} \int_\Sigma \zeta_\sigma ( R) \; \lambda \, dx_1dx_2, 
\end{align*}
where in the second inequality we have used that $\zeta' \leq 0$ and $|\nabla R|^2 \leq \frac{1}{2} |\nabla u|^2$, since $u$ is conformal.
Set $f(\sigma) = \int_\Sigma \zeta_\sigma (R) \, \lambda \, dx_1dx_2$. We have shown that
\[
       2 ( 1- O(\sigma^2)) f(\sigma) \leq \sigma f'(\sigma).
\]
Integrating this, we conclude that there exist $c>0$ such that the function
\begin{equation} \label{monotone}
       \sigma \mapsto \frac{e^{c\sigma^2} f(\sigma)}{\sigma^2}
\end{equation} 
is increasing for all $0 < \sigma < \sigma_0$. Approximating the characteristic function of $[-1,1]$, and letting $\zeta$ be the restriction to $\mathbb R^+$, it then follows that 
\[
      \frac{e^{c\sigma^2} A(u(\Sigma) \cap \mathcal{B}_\sigma(u(p)))}{\sigma^2}
\]
is increasing in $\sigma$ for $0<\sigma<\sigma_0$.

Since $\lambda=\frac{1}{2}|\nabla u|^2 \in L^1(\Sigma,\mathbb{R})$, 
\begin{equation} \label{density}
         \lim_{r \rightarrow 0} \frac{ \int_{B_r(x)} \lambda \, dx_1dx_2}{\pi r^2} = \lambda(x), 
         \quad \mbox{ a.e. } x \in \Sigma
\end{equation}
by the Lebesgue-Besicovitch Differentiation Theorem. 
Since $u$ is conformal, for every $\omega \in \mathbb{S}^1$,
\begin{equation} \label{directional}
        \lambda(x) =  \lim_{t \rightarrow 0} \frac{d^2(u(x+t\omega),u(x))}{t^2},
        \quad \mbox{ a.e. } x \in \Sigma
\end{equation}
(\cite[Theorem 1.9.6 and Theorem 2.3.2]{korevaar-schoen1}). Since $u$ is locally Lipschitz \cite[Theorem 1.2]{BFHMSZ}, by an argument as in the proof of Rademacher's Theorem (\cite[p. 83-84]{evans-gariepy}), 
\begin{equation}  \label{derivative}
        \lambda(x) =  \lim_{y \rightarrow x} \frac{d^2(u(y),u(x))}{|y-x|^2}
\end{equation}
for almost every $x \in \Sigma$.
To see this, choose $\{\omega_k\}_{k=1}^\infty$ to be a countable, dense subset of $\mathbb{S}^1$. Set
\[
      S_k=\{ x \in \Sigma \, : \lim_{t \rightarrow 0} \frac{ d(u(x+t\omega_k), u(x))}{t} \mbox{ exists, and is equal to}\, 
      \sqrt{\lambda(x)}  \}
\]
for $k=1, 2, \ldots$ and let
\[
         S=\cap_{k=1}^\infty S_k.
\]
Observe that $\mathcal{H}^2(\Sigma \setminus S)=0$. Fix $x \in S$, and
let $\varepsilon >0$. Choose $N$ sufficiently large such that if $\omega \in \mathbb{S}^1$ then
\[
       |\omega - \omega_k| < \frac{\varepsilon}{2 \mbox{Lip}(u)}
\]
for some $k \in \{ 1, \ldots, N\}$. Since
\[
     \lim_{t \rightarrow 0} \frac{ d(u(x+t\omega_k), u(x))}{t} = \sqrt{\lambda(x)}
\]
for $k=1, \ldots, N$, there exists $\delta >0$ such that if $|t|<\delta$ then
\[
       \left| \frac{d(u(x+t\omega_k),u(x))}{t} - \sqrt{\lambda(x)} \right| < \frac{\varepsilon}{2}
\]
for $k=1, \ldots, N$. Consequently, for each $\omega \in \mathbb{S}^1$ there exists $k \in \{1, \ldots, N\}$ such that 
\begin{align*}
      & \left| \frac{d(u(x+t\omega),u(x))}{t} - \sqrt{\lambda(x)} \right| \\
      & \leq \left| \frac{d(u(x+t\omega_k),u(x))}{t} - \sqrt{\lambda(x)}\right|
      +\left| \frac{d(u(x+t\omega),u(x))}{t} -\frac{d(u(x+t\omega_k),u(x))}{t} \right| \\
      & \leq \left|\frac{d(u(x+t\omega_k),u(x))}{t} - \sqrt{\lambda(x)} \right|
      +\left| \frac{d(u(x+t\omega),u(x+t\omega_k))}{t} \right| \\
      & < \frac{\varepsilon}{2} + \mbox{Lip}(u) |\omega -\omega_k| \\
      & < \varepsilon.
\end{align*}
Therefore the limit in (\ref{derivative}) exists, and (\ref{derivative}) holds, for almost every $x \in \Sigma$.
    
The zero set of $\lambda$ is of Hausdorff dimension zero by \cite{mese}. At points where $\lambda(x) \neq 0$ and (\ref{derivative}) holds, we have that for any $\varepsilon >0$
\[
      \mathcal{B}_{(1-\varepsilon) \sqrt \lambda \, r}(u(x)) \subset u(B_r(x)) 
      \subset \mathcal{B}_{(1+\varepsilon) \sqrt \lambda \, r}(u(x))
\]
if $r$ is sufficiently small. Therefore by (\ref{density}),
\begin{equation} \label{limit}
         \Theta(x):=\lim_{\sigma \rightarrow 0} \frac{ A(u(\Sigma) \cap \mathcal{B}_\sigma(u(x)))}{\pi \sigma^2} = 1,
         \quad \mbox{ a.e. } x \in \Sigma.
\end{equation}
By the monotonicity of (\ref{monotone}), $\Theta(x)$ exists for every $x \in \Sigma$, and $\Theta(x)$ is upper semicontinuous since it is a limit of continuous functions (the density at a given radius is a continuous function of $x$). Therefore, $\Theta(x) \geq 1$ for every $x \in \Sigma$. Together with the monotonicity of (\ref{monotone}), it follows that
\[
        A(u(\Sigma) \cap \mathcal{B}_\sigma(u(p))) \geq C\sigma^2
\]          
for $0<\sigma<\sigma_0$.
\end{proof}

\begin{remark}
Note that if $u: M \rightarrow \mathcal{B}_\rho(P)$ is a harmonic map from a compact Riemannian manifold $M$, then $u$ must be constant. This follows from the maximum principle, since equation (\ref{inequality}) implies that $R^2(x)=d^2(u(x),P)$ is subharmonic.
\end{remark}

For a {\em conformal} harmonic map from a surface into a Riemannian manifold, continuity follows easily using monotonicity (\cite[Theorem 10.4]{schoen}, \cite{gruter}, \cite[Theorem 9.3.2]{jost}). By Theorem \ref{monotonicity}, using this idea we can prove the following removable singularity result for conformal harmonic maps into a CAT(1) space.

\begin{theorem}[Removable singularity] \label{singularity}
If $u: \Sigma \setminus \{p\} \rightarrow X$ is a conformal harmonic map of finite energy, then $u$ extends to a locally Lipschitz harmonic map $u: \Sigma \rightarrow X$.
\end{theorem}

\begin{proof}
Let $B_r$ denote $B_r(p)$, the geodesic ball of radius $r$ centered at the point $p$ in $\Sigma$, and let $C_r=\partial B_r$ denote the circle of radius $r$ centered at $p$. By the Courant-Lebesgue Lemma, there exists a sequence $r_i \searrow 0$ so that 
\[
      L_i=L(u(C_{r_i})):=\int_{C_{r_i}} \sqrt{\lambda} \, ds_g \rightarrow 0
\]
as $i \rightarrow \infty$, where $ds_g$ denotes the induced measure on $C_{r_i}=\partial B_{r_i}$ from the metric $g$ on $\Sigma$. Since $E(u) < \infty$, $\lambda = \frac{1}{2} |\nabla u|^2$ is an $L^1$ function and, by the Dominated Convergence Theorem, 
\[
      A_i=A(u(B_{r_i} \setminus \{p\})):=\int_{B_{r_i} \setminus \{p\}} \lambda \, d\mu_g \rightarrow 0 
\]
as $i \rightarrow \infty$.

First we claim that there exists $P \in X$ such that $u(C_{r_i}) \rightarrow P$ with respect to the Hausdorff distance as $i \rightarrow \infty$. Let $d_{i,j}=d( u(C_{r_i}), u(C_{r_j}))$. Suppose $i < j$ so $r_i > r_j$, and choose $Q \in u(B_{r_i} \setminus \bar{B}_{r_j})$ such that $d(Q, u(C_{r_i}) \cup u(C_{r_j})) \geq d_{i,j}/2$. For $\sigma =\min\{\frac{d_{i,j}}{3},\frac{\rho}{2}\}$, by monotonicity (Theorem \ref{monotonicity}),
\[
      A(u(B_{r_i} \setminus \bar{B}_{r_j}) \cap \mathcal{B}_\sigma(Q)) \geq C \sigma^2.
\]
Since $A(u(B_{r_i} \setminus \bar{B}_{r_j}) \cap \mathcal{B}_r (Q)) \leq A(u(B_{r_i} \setminus \{p\}))=A_i$, it follows that $\sigma \leq c \sqrt{A_i}  \rightarrow 0$ as $i \rightarrow \infty$, and we must have $d_{i,j} \rightarrow 0$. Therefore any sequence of points $P_i \in u(C_{r_i})$ is a Cauchy sequence since
\[
     d(P_i,P_j) \leq d_{i,j} + L_i + L_j \rightarrow 0
\]
as $i,\, j \rightarrow \infty$. Hence, there exists $P \in X$ independent of the sequence, such that $P_i \rightarrow P$.

Finally, we claim that $\lim_{x \rightarrow p} u(x) =P$. It follows from this that  we may extend $u$ continuously to $\Sigma$ by defining $u(p)=P$. To prove the claim, consider a sequence $x_i \in \Sigma \setminus\{p\}$ such that $x_i \rightarrow p$. We want to show that $u(x_i) \rightarrow P$. Suppose $x_i \in B_{r_{j(i)}} \setminus \bar{B}_{r_{j(i)+1}}$ for some $j(i)$, and let $d_i=d(u(x_i), u(C_{r_{j(i)}}) \cup u(C_{r_{j(i)+1}}))$. For $\sigma =\min\{ \frac{d_i}{3},\frac{\rho}{2}\}$, by monotonicity (Theorem \ref{monotonicity}),
\[
        A(u(B_{r_{j(i)}} \setminus \bar{B}_{r_{j(i)+1}}) \cap \mathcal{B}_\sigma(u(x_i))) \geq C \sigma^2.
\]
Therefore, $\sigma< c \sqrt{A_{j(i)}} \rightarrow 0$ as $i \rightarrow \infty$, and we must have $d(u(x_i), u(C_{r_{j(i)}}) \cup u(C_{r_{j(i)+1}})) \rightarrow 0$. It follows that $u(x_i) \rightarrow P$ and $u$ extends continuously to $\Sigma$. 

We may now apply Theorem~\ref{theorem:minimizer} to show that $u$ is energy minimizing at $p$. Since $u$ is continuous, there exists $\delta >0$ such that $u(B_\delta) \subset \mathcal{B}_\rho(Q) \subset X$. By Theorem \ref{theorem:minimizer}, $u$ is the unique energy minimizer in $W_u^{1,2}(B_\delta, \mathcal{B}_\rho(Q))$. Hence $u$ is locally energy minimizing on $\Sigma$ and by \cite[Theorem 1.2]{BFHMSZ}, $u$ is locally Lipschitz on $\Sigma$.
\end{proof}

The following is derived using only domain variations as in \cite[Lemma 1.1]{schoen} (using \cite[Theorem 2.3.2]{korevaar-schoen1} to justify the computations involving change of variables) and is independent of the curvature of the target space (see for example, \cite[(2.3) page 193]{gromov-schoen}).

\begin{lemma} \label{holomorphic}
Let $u: \Sigma \rightarrow X$ be a harmonic map from a Riemann surface into a CAT(1) space. The Hopf differential
\[
     \Phi(z)=\left[ \pi \left( \frac{\partial}{\partial x_1},\frac{\partial}{\partial x_1} \right)  
                        - \pi \left( \frac{\partial}{\partial x_2},\frac{\partial}{\partial x_2} \right) 
                      -2i \pi \left( \frac{\partial}{\partial x_1},\frac{\partial}{\partial x_2}\right) \right] \, dz^2,
\]
where $z=x_1+ix_2$ is a local complex coordinate on $\Sigma$ and $\pi$ is the pull-back inner product, is holomorphic. 
\end{lemma}

\begin{corollary} \label{removable-singularity}
Let $u:\mathbb{C} \rightarrow X$ be a harmonic map of finite energy, then  $u$ extends to a locally Lipschitz harmonic map $u:\mathbb{S}^2 \rightarrow X$.  
\end{corollary}

\begin{proof}
Let $p: \mathbb{S}^2 \setminus \{n\} \rightarrow \mathbb{R}^2$ be stereographic projection from the north pole $n \in \mathbb{S}^2$. Set $\hat{u}=u \circ p : \mathbb{S}^2 \setminus \{n\} \rightarrow X$. We will show that $n$ is a removable singularity.

Let $\varphi=\pi ( \frac{\partial}{\partial x_1},\frac{\partial}{\partial x_1} )  
                        - \pi ( \frac{\partial}{\partial x_2},\frac{\partial}{\partial x_2} ) 
                      -2i \pi( \frac{\partial}{\partial x_1},\frac{\partial}{\partial x_2})$.
By Lemma \ref{holomorphic}, the Hopf differential $\Phi(z)=\varphi(z) dz^2$ is holomorphic on $\mathbb{C}$.                       
By assumption, 
\[
     E(u) = \int_{\mathbb{R}^2} \left( \|u_*(\frac{\partial}{\partial x_1}) \|^2 
                  + \|u_*(\frac{\partial}{\partial x_2}) \|^2 \right) \; dx_1 dx_2 < \infty
\]
and therefore
\[
     \int_{\mathbb{R}^2} | \varphi| \; dx_1 dx_2 \leq 2E(u) < \infty.
\]
Thus $|\varphi| \in L^1(\mathbb{C},\mathbb{R})$ and is subharmonic, and hence $\varphi \equiv 0$ and $u$ is conformal. Then by Theorem \ref{singularity}, $u$ extends to a locally Lipschitz harmonic map $u: \mathbb{S}^2 \rightarrow X$.                  
\end{proof}

\section{Harmonic Replacement Construction} \label{section:harmonic-replacement}

In this section we prove the main theorem:

\begin{theorem} 
Let $\Sigma$ be a compact Riemann surface, $X$ a compact locally CAT(1) space and $\varphi \in C^0 \cap W^{1,2}(\Sigma,X)$. Then either there exists a harmonic map $u:\Sigma \rightarrow X$ homotopic to $\varphi$ or a conformal harmonic map $v:\mathbb S^2 \rightarrow X$.
\end{theorem}

\begin{lemma}[Jost's covering lemma, {\cite{jost} Lemma 9.2.6}] \label{jostcovering}
For a compact Riemannian manifold $\Sigma$, there exists $\Lambda = \Lambda(\Sigma) \in \N$  with the following property:  for any covering 
\[
\Sigma \subset \bigcup_{i=1}^m B_r(x_i)
\]
by open balls, there exists a partition $I^1, \dots I^{\Lambda}$ of the integers $\{1, \dots, m\}$ such that for any $l \in \{1, \dots, \Lambda\}$ and two distinct elements $i_1, i_2$ of $I^l$,
\[
B_{2r}(x_{i_1}) \cap B_{2r}(x_{i_2}) = \emptyset.
\]
\end{lemma}

\begin{definition} \label{jc}
\emph{
For each $k=0,1,2, \dots$,  we fix a covering 
\[
{\mathcal O}_k=\{ B_{2^{-k}}(x_{k,i}) \}_{i=1}^{m_k}
\]
 of $\Sigma$ by balls of radius $2^{-k}$.  Furthermore,  let $I^1_k, \dots, I^{\Lambda}_k$ be the disjoint subsets of $\{1, \dots, m_k\}$ as in Lemma~\ref{jostcovering}; in other words, for every $l \in \{1, \dots, \Lambda\}$, 
 \begin{equation} \label{disjoint}
B_{2^{-k+1}}(x_{k,i_1}) \cap B_{2^{-k+1}}(x_{k,i_2}) = \emptyset, \ \ \forall i_1, i_2 \in I^l_k, \ \ i_1 \neq i_2.
\end{equation}
  By the Vitali Covering Lemma, we can assure that }
 \begin{equation} \label{vitali}
B_{2^{-k-3}}(x_{k,i_1}) \cap B_{2^{-k-3}}(x_{k,i_2})=\emptyset, \ \ \forall i_1, i_2 \in \{1,\dots,m_k\}, \ \ i_1\neq i_2.
 \end{equation}
\end{definition}

Let $\Sigma$ be a compact Riemann surface. By uniformization, we can endow  $\Sigma$ with a Riemannian metric of constant Gaussian curvature $+1$, $0$ or $-1$.  Let $\Lambda=\Lambda(\Sigma)$ be as in Lemma~\ref{jostcovering} and  $\rho=\rho(X)>0$ be as in Lemma~\ref{dir}.   We inductively define a sequence of numbers 
\[
\{r_n\} \subset 2^{-\N}:=\{1,2^{-1}, 2^{-2}, \dots \}\]
 and a sequence of  finite energy maps 
 \[
 \{u_n^l:\Sigma \rightarrow X \}
 \]
  for $l=0, \dots, \Lambda$, $n=1, \dots, \infty$ as follows:  \\

{\sc Initial Step 0:}  Fix $\kappa_0 \in \N$  such that  $B_{2^{-\kappa_0}}(x)$ is homeomorphic to a disk  for all $x \in \Sigma$.
Let $u_0^0:=\varphi \in C^0 \cap W^{1,2}(\Sigma, X)$, and let
\[
r_0'=\sup\{r>0:  \forall \, x \in \Sigma, \exists \, P \in X \mbox{ such that } u_0^0(B_{2r}(x)) \subset \B_{3^{-\Lambda} \rho}(P) \}
\]
and $k_0'>0$ be such that
\[
2^{-k_0'} \leq r_0' < 2^{-k_0'+1}.
\]
Define
\[
r_0=2^{-k_0}=\min\{2^{-k_0'}, 2^{-{\kappa_0}}\},
\]
and let
\[
{\mathcal O}_{k_0}=\{ B_{r_0}(x_{{k_0},i})\}_{i=1}^{m_{k_0}}
\mbox{ and }
 I^1_{k_0}, \dots, I^{\Lambda}_{k_0}
 \]
be as in Definition~\ref{jc}.
For $l \in \{1, \ldots, \Lambda\}$ inductively define $u_0^l: \Sigma \rightarrow X$ from $u_0^{l-1}$ by setting
\[
u_0^l= \left\{
\begin{array}{ll}
u_0^{l-1} & \mbox{ in }\Sigma \backslash \bigcup_{i \in I_{k_0}^l} B_{2r_0}(x_{_{k_0},i})\\
^{Dir}u_0^{l-1}& \mbox{ in }B_{2r_0}(x_{_{k_0},i}), \  i \in I^l_{k_0}
\end{array}
\right.
\]
where
$^{Dir}u_0^{l-1}$ is the unique Dirichlet solution in  $W^{1,2}_{u_0^{l-1}}(B_{2r_0}(x_{k_0,i}),\mathcal{B}_{\rho}(P))$ of Lemma~\ref{dir}. Here there are two things to check related to the definition of the Dirichlet solution. First, since $B_{2r_0}(x_{k_0,i_1}) \cap B_{2r_0}(x_{k_0,i_2}) = \emptyset, \ \forall \; i_1, i_2 \in I^l_{k_0}$ with $i_1 \neq i_2$ (cf. (\ref{disjoint})), there is no issue of interaction between solutions at a single step so the map is well-defined if it exists. Second, we claim that $u_0^{l-1}(B_{2r_0}(x_{k_0,i})) \subset \mathcal{B}_{3^{-\Lambda+(l-1)}\rho}(P) \subset \mathcal{B}_\rho(P)$ for some $P \in X$ and thus the Dirichlet solution exists and is unique by Lemma \ref{dir}. To verify the claim, first note that for each $i=1, \ldots , m_{k_0}$ there exists $P \in X$ such that 
$u_0^1(B_{2r_0}(x_{k_0,i})) \subset \mathcal{B}_{3^{-\Lambda+1}\rho}(P)$. Indeed, if 
$B_{2r_0}(x_{k_0,i}) \cap B_{2r_0}(x_{k_0,j}) = \emptyset$ for all $j \in I_{k_0}^1$ then $u_0^1=u_0^0$ on $B_{2r_0}(x_{k_0,i})$ and so 
$u_0^1(B_{2r_0}(x_{k_0,i}))=u_0^0(B_{2r_0}(x_{k_0,i}))\subset \mathcal{B}_{3^{-\Lambda}\rho}(P)$ for some $P$. On the other hand, if $B_{2r_0}(x_{k_0,i}) \cap B_{2r_0}(x_{k_0,j}) \neq \emptyset$ for one or more $j \in I_{k_0}^1$, then since $u_0^0(B_{2r_0}(x_{k_0,i})) \subset \mathcal{B}_{3^{-\Lambda}\rho}(P)$ 
for some $P$ and $u_0^1(B_{2r_0}(x_{k_0,j})) \subset \mathcal{B}_{3^{-\Lambda}\rho}(P_j)$ for some $P_j$ with $\mathcal{B}_{3^{-\Lambda}\rho}(P) \cap \mathcal{B}_{3^{-\Lambda}\rho}(P_j) \neq \emptyset$, it follows that $u_0^1(B_{2r_0}(x_{k_0,i})) \subset \mathcal{B}_{3^{-\Lambda+1}\rho}(P)$. Inductively, 
we may show that for each $i=1, \ldots , m_{k_0}$ and $l \in \{1, \dots, \Lambda\}$ there exists $P \in X$ such that $u_0^{l-1}(B_{2r_0}(x_{k_0,i})) \subset \mathcal{B}_{3^{-\Lambda+(l-1)}\rho}(P)$, as claimed. 
 \\

{\sc Inductive Step }$n$:   Having defined 
\[
 r_0, \dots, r_{n-1} \in 2^{-\N},
 \] 
and
\[
u_\nu^0, \, u_{\nu}^1, \dots, u_{\nu}^{\Lambda}: \Sigma \rightarrow X, \ \  \nu=0,1, \dots, n-1,
\]
we set $u_n^0=u_{n-1}^{\Lambda}$ and  define
\[
r_n \in 2^{-\N}  \mbox{ and } \ u_n^1, \dots, u_n^{\Lambda}
\]
  as follows.  Let
\[
r_n'=\sup\{r>0:  \forall \, x \in \Sigma, \exists \, P \in X \mbox{ such that } u_n^0(B_{2r}(x)) \subset \B_{3^{-\Lambda}\rho}(P) \}
\]
and $k_n' \in \N$ be such that
\[
2^{-k_n'} \leq r_n' < 2^{-k_n'+1}.
\]
Define
\[
r_n  = 2^{-k_n}= \min\{2^{-k_n'}, 2^{-\kappa_0} \}.
\]
Let 
\[
{\mathcal O}_{k_n}=\{ B_{r_n}(x_{k_n,i})\}_{i=1}^{m_{k_n}} \mbox{ and } 
 I^1_{k_n}, \dots, I^{\Lambda}_{k_n}
  \]
be as in Definition~\ref{jc}.
Having defined $u_n^0, \dots, u_n^{l-1}$, we now define $u_n^l:\Sigma \rightarrow X$ by setting
\[
u_n^l= \left\{
\begin{array}{ll}
u_n^{l-1} & \mbox{ in }\Sigma \backslash \bigcup_{i \in I^l_{k_n}} B_{2r_n}(x_{k_n,i})\\
^{Dir}u_n^{l-1} & \mbox{ in }B_{2r_n}(x_{k_n,i}), \  i \in I^l_{k_n}
\end{array}
\right.
\]
where $^{Dir}u_n^{l-1}$ is the unique Dirichlet solution in  $W^{1,2}_{u_n^{l-1}}(B_{2r_n}(x_{k_n,i}),\mathcal{B}_\rho(P))$ for some $P$ of Lemma~\ref{dir}. 
\vskip 5mm

This completes the inductive construction of the sequence $\{u_n^l\}$.  Note that 
\[
E(u_n^{\Lambda}) \leq \dots \leq E(u_n^0) = E(u_{n-1}^{\Lambda}), \ \ \forall \, n=1,2, \dots.
\]
Thus, there exists  $E_0$ such that 
\begin{equation} \label{Enot}
\lim_{n \rightarrow \infty} E(u^l_n) = E_0, \ \ \ \forall \, l=0, \dots, \Lambda.
\end{equation}

\vskip10mm

We consider the following two cases separately: 
\begin{itemize}
\item[ ]{\bf CASE 1}:  $\liminf_{n\rightarrow \infty} r_n>0$.
\item[ ]{\bf CASE 2}:  $\liminf_{n\rightarrow \infty} r_n=0$.\\
\end{itemize}

\noindent For {\bf CASE 1}, we prove that there exists a harmonic map $u:\Sigma \rightarrow X$ homotopic to $\varphi=u_0^0$.   We will need the following two claims.

\begin{claim}  \label{l2close} 
\emph{For any $l \in \{0,\dots \Lambda-1\}$, }
\[
 \lim_{n \rightarrow \infty} ||d(u_{n}^l,u_{n}^{\Lambda})||_{L^2(\Sigma)}=  0.
 \]
\end{claim}
In particular, letting $l=0$ in Claim~\ref{l2close}, we obtain
\begin{equation} \label{l2closeagain0CASE1}
\lim_{n \rightarrow \infty} ||d(u_{n-1}^{\Lambda},u_{n}^{\Lambda})||_{L^2(\Sigma)}=  0.
\end{equation}

\begin{proof}
Fix $l \in \{0,\dots, \Lambda-1\}$.   For $n \in \N$,  $\lambda \in \{l+1, \dots, \Lambda\}$ and   $i \in I^{\lambda}_{k_n}$, we apply Theorem~\ref{theorem:energy-convexity} with $u_0=u_{n}^{\lambda-1}\big|_{B_{2r_n}(x_{k_n,i})}$, $u_1=u_{n}^\lambda \big|_{B_{2r_n}(x_{k_n,i})}$ and $\Omega=B_{2r_n}(x_{k_n,i})$.  Let $w:\Sigma \rightarrow X$ be the map defined as $w=u_{n}^\lambda=u_{n}^{\lambda-1}$ outside $\bigcup_{i \in I^\lambda_{k_n}} B_{2r_n}(x_{k_n,i})$ and  the map corresponding to $w$ in  Theorem~\ref{theorem:energy-convexity} in each  $B_{2r_n}(x_{k_n,i})$.  Then
\begin{align*}
   (\cos^8 \rho) & \int_{B_{2r_n}(x_{k_n,i})}  
   \left| \nabla \frac{\tan \frac{1}{2}d(u_{n}^{\lambda-1},u_{n}^\lambda)}{\cos R} \right|^2 d\mu \\
   &\leq \frac{1}{2} \left(\int_{B_{2r_n}(x_{k_n,i})} | \nabla u_n^{\lambda-1}|^2d\mu  
   + \int_{B_{2r_n}(x_{k_n,i})} | \nabla u_n^\lambda|^2  d\mu\right)  - \int_{B_{2r_n}(x_{k_n,i})} | \nabla w|^2 d\mu.
\end{align*}
Summing over $i$, using that $w=u_{n}^\lambda=u_{n}^{\lambda-1}$ outside $\bigcup_{i \in I^\lambda_{k_n}} B_{2r_n}(x_{k_n,i})$, and applying the Poincar\'{e} inequality, we obtain
\[
\int_{\Sigma} d^2(u_{n}^{\lambda-1},u_{n}^\lambda) d\mu
 \leq   C \left(  \frac{1}{2} E(u_{n}^{\lambda-1}) + \frac{1}{2} E(u_{n}^\lambda) -E(w) \right),
\]
where here and henceforth $C$ is a constant independent of $n$.
Since  $u_{n}^\lambda$ is harmonic in $\bigcup_{i \in I^\lambda_{k_n}} B_{2r_n}(x_{k_n,i})$,  we have  $E(u_{n}^\lambda) \leq E(w)$.  Hence
\[
\int_{\Sigma} d^2(u_{n}^{\lambda-1},u_{n}^\lambda) d\mu \leq C \left(  \frac{1}{2} E(u_{n}^{\lambda-1}) - \frac{1}{2} E(u_{n}^\lambda)  \right).
\]
Thus,  
\begin{eqnarray*}
 \int_{\Sigma} d^2(u_{n}^l,u_{n}^{\Lambda}) d\mu 
& \leq & \int_{\Sigma} \left(\sum_{\lambda=l+1}^{\Lambda} d(u_{n}^{\lambda-1},u_{n}^{\lambda}) \right)^2 d\mu \nonumber \\
& \leq &(\Lambda-l)^2\sum_{\lambda=l+1}^{\Lambda} \int_{\Sigma} d^2 (u_{n}^{\lambda-1},u_{n}^{\lambda})d\mu\nonumber \\
& \leq & C \sum_{\lambda=l+1}^{\Lambda} \left( E(u_{n}^{\lambda-1}) - E(u_{n}^\lambda)\right)\nonumber \\
& = & C  \left( E(u_{n}^l) - E(u_{n}^{\Lambda})\right). \label{pn}
\end{eqnarray*}
\\
This proves the claim since $\lim_{n \rightarrow \infty} \left( E(u_{n}^l) - E(u_{n}^{\Lambda}) \right)=0$ by (\ref{Enot}).  
\end{proof}

\begin{claim} \label{epsilondelta}
Let $\epsilon>0$ such that $3^{-\Lambda} \epsilon \leq \rho$, $l \in \{0, 1, \dots, \Lambda\}$ and $n \in \N$ be given. If  $\delta \in (0,r_n)$ 
is such that
\begin{equation} \label{choiceofdelta}
\sqrt{\frac{8\pi E(u_0^0)}{\log \delta^{-2}}} \leq 3^{-\Lambda} \epsilon,
\end{equation}
then
\[
\forall \, x \in \bigcup_{\lambda=1}^{l} \bigcup_{i \in I^{\lambda}_{k_n}}B_{r_n}(x_{k_n,i}), \  \ \exists \, P \in X \mbox{ such that } u_{n}^l(B_{\delta^{\Lambda}}(x)) \subset \B_{3\epsilon}(P).
\]
In particular, for $l =\Lambda$, 
$\forall \, x \in \Sigma, \  \exists \, P \in X \mbox{ such that } u_{n}^\Lambda(B_{\delta^{\Lambda}}(x)) \subset \B_{3\epsilon}(P)$.
\end{claim}

\begin{proof}
Fix $\epsilon$, $l$, $n$ and let $\delta$ be as in (\ref{choiceofdelta}).  For  $x \in \bigcup_{\lambda=1}^{l} \bigcup_{i \in I^{\lambda}_{k_n}}B_{r_n}(x_{k_n,i})$, there exists  $\lambda \in \{1, \dots, l\}$ such that $x \in B_{r_n}(x_{k_n,i})$ for some $i \in I^{\lambda}_{k_n}$ and hence 
\[
B_{r_n}(x) \subset B_{2r_n}(x_{k_n,i}).
\]
Since $u_{n}^{\lambda}$ is harmonic in $B_{2r_n}(x_{k_n,i})$,  it is harmonic in $B_{r_n}(x)$.    
By the Courant-Lebesgue Lemma, there exists 
\[
R_1(x) \in (\delta^2,\delta)
\]
such that
\[
u_{n}^{\lambda}(\partial B_{R_1(x)}(x)) \subset \B_{3^{-\Lambda} \epsilon}(P_1) \mbox{ for some }P_1 \in X.
\]
Since $u_{n}^{\lambda}$ is a Dirichlet solution and $R_1(x) < \rho$, by Lemma \ref{dir}
\[
u_{n}^{\lambda}(B_{\delta^2}(x))  \subset u_{n}^{\lambda}(B_{R_1(x)}(x)) \subset \B_{3^{-\Lambda} \epsilon}(P_1).
\]

Next, by the Courant-Lebesgue Lemma, there exists 
\[
R_2(x) \in (\delta^3, \delta^2)
\]
such that
\begin{equation} \label{CourantLebesgueR2}
u_{n}^{\lambda+1}(\partial B_{R_2(x)}(x)) \subset \B_{3^{-\Lambda} \epsilon}(P_2') \mbox{ for some }P_2' \in X.
\end{equation}

There are two cases to consider:\\ 
\\
\emph{Case a}. $B_{R_2(x)}(x) \cap  \overline{\bigcup_{i \in I^{\lambda+1}_{k_n}}  B_{2r_n}(x_{k_n,i}) }= \emptyset$.   In this case, $u_{n}^{\lambda+1} =u_{n}^{\lambda}$ in $B_{R_2(x)}(x)$. Since $u_{n}^{\lambda}$ is harmonic on this ball,
\[
 u_{n}^{\lambda+1}(B_{R_2(x)}(x)) = u_{n}^{\lambda}(B_{R_2(x)}(x)) \subset u_{n}^{\lambda}(B_{\delta^2}(x)) \subset \B_{3^{-\Lambda}\epsilon}(P_1).
\]
In this case we let $P_2=P_1$. \\
\\
\emph{Case b}. $B_{R_2(x)}(x) \cap  {\bigcup_{i \in I^{\lambda+1}_{k_n}}  B_{2r_n}(x_{k_n,i}) }\neq \emptyset$.  In this case, $u_{n}^{\lambda+1}$ is only piecewise harmonic on $B_{R_2(x)}(x)$.  The regions of harmonicity are of two types. On the region $\Omega:= B_{R_2(x)}(x)\backslash \overline{\bigcup_{i \in I^{\lambda+1}_{k_n}}  B_{2r_n}(x_{k_n,i}) }$, we have $u_{n}^{\lambda+1} =u_{n}^{\lambda}$.  As in \emph{Case a},   we conclude that the image of this region is contained in $B_{3^{-\Lambda}\epsilon}(P_1)$. All other regions, which we index $\Omega_i$, have two smooth boundary components, one on the interior of $B_{R_2(x)}(x)$, which we label $\gamma_i$, and one on $\partial B_{R_2(x)}(x)$, which we label $\beta_i$. By construction $u_{n}^{\lambda+1} =u_{n}^{\lambda}$ on $\gamma_i$, thus
\[
u_{n}^{\lambda+1}(\gamma_i) \subset \B_{3^{-\Lambda}\epsilon}(P_1).
\]
Moreover,  $u_{n}^{\lambda+1}(\beta_i) \subset \B_{3^{-\Lambda}\varepsilon}(P_2')$ by (\ref{CourantLebesgueR2}). 
Notice that in this case,
\[
\B_{3^{-\Lambda}\epsilon}(P_1) \cap \B_{3^{-\Lambda}\epsilon}(P_2') \neq \emptyset.
\]  
Thus,  by the triangle inequality there exists $P_2 \in X$ such that
\[
u_{n}^{\lambda+1}(\cup_{i \in I^{\lambda+1}_{k_n}}\partial\Omega_i)  \subset \B_{3^{-\Lambda+1}\epsilon}(P_2).
\]
Since $u_{n}^{\lambda+1}$ is harmonic on each $\Omega_i$, 
\[
u_{n}^{\lambda+1}(\cup_{i \in I^{\lambda+1}_{k_n}}\Omega_i) \subset \B_{3^{-\Lambda+1}\epsilon}(P_2).
\]
Since $\overline{B_{R_2(x)}(x)}=\overline \Omega \cup \bigcup_{i \in I^{\lambda+1}_{k_n}}\overline \Omega_i$,  
\[
u_{n}^{\lambda+1}(B_{R_2(x)}(x)) \subset \B_{3^{-\Lambda+1}\epsilon}(P_2).
\]
Thus, we have shown that  in either \emph{Case a} or \emph{Case b},
\[
 u_{n}^{\lambda+1}(B_{\delta^3}(x))\subset  u_{n}^{\lambda+1}(B_{R_2(x)}(x)) \subset \B_{3^{-\Lambda+1}\epsilon}(P_2).
\]
After iterating this argument for $u_n^{\lambda+2}, \dots, u_n^l$, we conclude that  there exists $P_{l-\lambda+1} \in X
$ such that
\[
u_{n}^l(B_{\delta^{\Lambda}}(x)) \subset u_{n}^l(B_{\delta^{l-\lambda+2}}(x)) \subset \B_{3^{-\Lambda+l-\lambda}\epsilon}(P_{l-\lambda+1}) \subset \B_{3\epsilon}(P_{l-\lambda+1}).
\]
Letting $P=P_{l-\lambda+1}$, we obtain the assertion of Claim~\ref{epsilondelta}.
\end{proof}

Next, note that since $\liminf_{n\rightarrow \infty} r_n>0$, applying Claim~\ref{epsilondelta}  with $\delta \in (0,\liminf r_n)$ and $l=\Lambda$ implies that 
\begin{equation} \label{equicontLambda}
\{u_{n}^{\Lambda}\} \mbox{ is an equicontinuous family of maps in } \Sigma.
\end{equation}
Combining  (\ref{l2closeagain0CASE1}) and (\ref{equicontLambda}),  we conclude that 
\begin{equation} \label{limitforcase1}
\exists \, u \in C^0(\Sigma, X) \mbox{ such that } u_n^{\Lambda} \rightrightarrows u.
\end{equation}
Indeed, from (\ref{equicontLambda}), there exists a subsequence 
$\{u_{n_{j}}^{\Lambda}\}$ of $\{u_{n}^{\Lambda}\}$  and a map $u$ such that
\[
u_{n_{j}}^{\Lambda} \rightrightarrows u.
\]
If the whole sequence $\{u_n^{\Lambda}\}$ does not converge to $u$, then we could have chosen the subsequence 
$\{u^\Lambda_{n_{j}}\}$ in such a way  that $\{u^\Lambda_{n_j-1}\}$ does not converge uniformly to $u$.  By taking a further subsequence of $\{u^\Lambda_{n_j-1}\}$ if necessary, we can then assume that there exists a continuous map $v \neq u$ such that
\[
u^\Lambda_{n_j-1} \rightrightarrows v.
\]
  On the other hand,  (\ref{l2closeagain0CASE1}) asserts that
\[
||d(u^\Lambda_{n_j-1}, u^\Lambda_{n_j})||_{L^2(\Sigma)} \rightarrow 0.
\]
Since
\[
||d(u,v)||_{L^2(\Sigma)} \leq  ||d(u,u^\Lambda_{n_{j}})||_{L^2(\Sigma)} + ||d(u^\Lambda_{n_{j}},u^\Lambda_{n_j-1})||_{L^2(\Sigma)} + ||d(u^\Lambda_{n_j-1},v)||_{L^2(\Sigma)}
\]
and the right hand side goes to 0 as $j \rightarrow \infty$, we conclude $u=v$.  This contradiction proves (\ref{limitforcase1}).

We will now prove that the limit map $u$ of (\ref{limitforcase1}) is a harmonic map.  
Since $\liminf_{n \rightarrow \infty} r_n>0$, there exist $k \in \N$ and an increasing sequence  $\{n_j\}_{j=1}^{\infty} \subset \N$ such that $r_{n_j}=2^{-k}$ (or equivalently $k_{n_j}=k$).  In particular,   the covering used for {\sc Step }$n_j$ in the inductive construction of $u_{n_j}^0, \dots, u_{n_j}^{\Lambda}$  is the same for all $j=1,2,\dots$.
Thus, we can use the following  notation for simplicity: 
\[
{\mathcal O}={\mathcal O}_{k_j}, I^l=I^l_{k_j}, \ B_i=B_{r_{n_j}}(x_{k_{n_j},i}) \mbox{ and }  \ 
tB_i=B_{tr_{n_j}}(x_{k_{n_j},i}) \mbox{ for } t \in \mathbb{R}^+.
\]

With this notation, Claim~\ref{epsilondelta} implies 
that for a fixed $l \in \{1, \dots, \Lambda\}$,
\begin{equation} \label{equicontCASE1}
\{u_{n_j}^l\} \mbox{ is an equicontinuous family of maps on } B^l:=\bigcup_{\lambda=1}^{l} \bigcup_{i \in I^{\lambda}}B_i.
\end{equation}
We claim that for every $l \in \{1, \dots, \Lambda\}$, 
\begin{equation} \label{unc0}
u_{n_j}^l \rightrightarrows u \mbox{ on }  B^l \mbox{ where $u$ is as in (\ref{limitforcase1})}.
\end{equation}
Indeed, if (\ref{unc0}) is not true, consider a subsequence of $\{u_{n_j}^l\}$  that does not converge to $u$.  By (\ref{equicontCASE1}), we can assume (by taking a further  subsequence if necessary) that 
\[
\exists \, v :   B^l  \rightarrow X \mbox{ such that } u_{n_j}^l \rightrightarrows v \neq u|_{B^l}.
\]
Combining this with  (\ref{limitforcase1}) and Claim~\ref{l2close}, we conclude that
\[
||d(v,u)||_{L^2(B^l)}=\lim_{j \rightarrow \infty} ||d(u_{n_j}^l,u_{n_j}^{\Lambda})||_{L^2(B^l)} \leq \lim_{j \rightarrow \infty} ||d(u_{n_j}^l,u_{n_j}^{\Lambda})||_{L^2(\Sigma)} =0
\]
which in turn implies  that $u=v$.  This   contradiction proves (\ref{unc0}).

Finally, we are ready to prove the harmonicity of $u$.  For an arbitrary point $x \in \Sigma$, there exists $l \in \{1, \dots, \Lambda\}$ and $i \in I^l$ such that $x \in B_i$.  Since $u^l_{n_j}$ is energy minimizing in $B_i$ and  $u^{l}_{n_j} \rightrightarrows u$ in $B_i$ by (\ref{unc0}),  Lemma~\ref{luckhaus} implies that
$u$ is energy minimizing in  $\frac{1}{2}B_i$.

The map $u$  is homotopic to $\varphi$ since it is a uniform limit of $u_{n_j}^{\Lambda}$ each of which is homotopic to  $\varphi$.   This completes the proof for {\bf CASE 1} as $u$ is the desired harmonic map homotopic to $\varphi$.

\vskip 8mm

\noindent For {\bf CASE 2}, we prove that there exists a non-constant harmonic map $u:\Sp^2 \rightarrow X$. \\

Recall that we have endowed  $\Sigma$ with a metric $g$ of constant Gaussian curvature that is identically $+1$, $0$ or $-1$. 
Fix 
\begin{equation*} 
y_* \in \Sigma
\end{equation*} and a local conformal chart 
\begin{equation*}
\pi:  U \subset \C \rightarrow \pi(U)=B_1(y_*) \subset \Sigma
\end{equation*}
such that 
\begin{equation*}
\pi(0)=y_*
\end{equation*}
and the metric  $g=(g_{ij})$   of $\Sigma$ expressed with respect to this local coordinates satisfies
\begin{equation} \label{deltaij}
 g_{ij}(0)=\delta_{ij}.
\end{equation}
For each $n$,   the definition of $r_n$ implies that we can find $y_n, y_n' \in \Sigma$ with
\[
2r_n \leq d_g(y_n,y_n') \leq 4r_n
\]
where $d_g$ is the  distance function on $\Sigma$ induced by the metric $g$, and
\[
d(u_n^0(y_n),u_n^0(y_n')) \geq 3^{-\Lambda}\rho.
\]
Since $\Sigma$ is a compact Riemannian surface of constant Gaussian curvature, there exists an isometry  $\iota_n: \Sigma \rightarrow \Sigma$ such that $\iota_n(y_*) =y_n$.
Define the conformal coordinate chart
\begin{equation*} 
\pi_n:  U \subset \C \rightarrow \pi_n(U)=B_1(y_n) \subset \Sigma, \ \ \ \ \ \ \pi_n(z):=\iota_n \circ \pi(z).
\end{equation*}
Thus,
\[
\pi_n(0)=y_n.
\] 
Define the dilatation map
\begin{equation*}
\Psi_n:\C \rightarrow \C, \ \ \Psi_n(z)=r_n z
\end{equation*}
and set $\Omega_n:=\Psi_n^{-1}\circ \pi_n^{-1}(B_1(y_n))\subset \C$ and
\begin{equation*} 
\tilde{u}_n^l:  \Omega_n \rightarrow X, \ \ \ \tilde{u}_n^l:=u_n^l \circ \pi_n \circ \Psi_n.
\end{equation*}
Since $\liminf_{n \rightarrow \infty} r_n =0$, there exists a subsequence 
\begin{equation} \label{subsequence}
\{r_{n_j}\} \ \mbox{ 
 such that } \lim_{j \rightarrow \infty} r_{n_j}=0.
 \end{equation}
   Thus, $\Omega_{n_j} \nearrow \C$.   Furthermore,  (\ref{deltaij}) implies that
\[
\lim_{j \rightarrow \infty} \frac{d_g(y_{n_j}',y_{n_j})}{|\pi_{n_j}^{-1}(y_{n_j}')|}=1.
\]
Hence, for $z_n = \Psi_n^{-1} \circ \pi_n^{-1}(y'_n)$, 
\begin{equation} \label{nonconst1}
2\leq \lim_{j \rightarrow \infty} |z_{n_j}| \leq 4
\end{equation}
and
\begin{equation} \label{nonconst2}
d(\tilde{u}_{n_j}^0(z_{n_j}),\tilde{u}_{n_j}^0(0))= d(u_{n_j}^0(y_{n_j}'),u_{n_j}^0(y_{n_j})) \geq 3^{-\Lambda} \rho.
\end{equation}
Additionally, by the conformal invariance of energy, we have that 
\begin{equation} \label{energybound}
E(\tilde{u}_n^l) =E(u_n^l\big|_{B_1(y_n)}) \leq E(u_0^0).
\end{equation}
%

For $R>0$, let 
\[
D_R:=\{z \in \C:  |z|<R\}.
\]
Since harmonicity is invariant under conformal transformations of the domain, we can follow {\bf CASE 1} (cf.~(\ref{l2closeagain0CASE1}),  (\ref{equicontLambda}) and (\ref{limitforcase1})) and  prove that
\[ 
||d(\tilde{u}_{n-1}^{\Lambda},\tilde{u}_n^{\Lambda})||_{L^2(D_R)} \rightarrow 0,
\] 
\[ 
\{\tilde{u}^\Lambda_{n}\}_{n=n_R}^{\infty}
\mbox{ is an equicontinuous family in } D_R
\] 
for some $n_R$, and
\begin{equation} \label{limitmapCASE2}
\exists \, \tilde{u}_R:D_R \rightarrow X \mbox{ \ such that \ }
\tilde{u}_n^{\Lambda} \rightrightarrows \tilde{u}_R \mbox{ in } D_R.
\end{equation}
Below, we will prove harmonicity of $\tilde{u}_R$ by following a similar proof to {\bf CASE 1}.    We first need the following lemma.

\begin{lemma} \label{scalingballs}
Let ${\mathcal O}_{k_n}$ be as in Definition~\ref{jc}. For a fixed $R>0$, there exists $M$ independent of $n$ such that for every $n \in \N$, 
\[
\left|\{i :  B_{2^{-k_n}}(x_{k_n,i}) \cap \left( \pi_n \circ \Psi_{n}(D_R)\right) \neq \emptyset\}\right|\leq M.
\]
\end{lemma}

\begin{proof} 
 By (\ref{deltaij}), 
\[
\lim_{n \rightarrow \infty} \frac{\mathrm{Vol}(\pi_n \circ \Psi_n(D_{2R}))}{4\pi R^22^{-2k_n}} = 1
\]
and 
\[
\lim_{n \rightarrow \infty} \frac{\mathrm{Vol}(B_{2^{-k_n-3}}(x_{n,i}))}{\pi 2^{-2k_n-6}}=1
\]
where $\mathrm{Vol}$ is the volume in $\Sigma$.
Let $J \subset \{1, \dots, m_{k_n}\}$ be such that
\[
J =\{i :  B_{2^{-k_n}}(x_{k_n,i}) \cap \left(\pi_n \circ \Psi_{n}(D_{R})\right) \neq \emptyset\}.
\]
By
 (\ref{vitali}), we have that for sufficiently large $k_n$, 
 \begin{eqnarray*}
|J| \pi 2^{-2k_n-6} & \leq & 2 \sum_{i \in J} \mathrm{Vol} (B_{2^{-k_n-3}}(x_{k_n,i})) \\
& \leq & 2\mathrm{Vol}(\pi_n \circ \Psi_n(D_{2R})) \\
& \leq & 16\pi R^22^{-2k_n}.
 \end{eqnarray*}
Hence $|J| \leq R^2 2^{10}$ and $\{B_{2^{-k_n}}(x_{k_n,i})\}_{i \in J}$ covers $D_R$.
   \end{proof}

For each $B_{2^{-k_{n}}}(x_{k_{n},i}) \in {\mathcal O}_{k_n}$, let
\[
\tilde{B}_{n,i}:=  \Psi_{n}^{-1} \circ \pi_n^{-1}(B_{2^{-k_{n}}}(x_{k_{n},i}))
 \]
 and
 \[
2\tilde{B}_{n,i}:=\Psi_{n}^{-1} \circ \pi_n^{-1}(B_{2^{-k_{n}+1}}(x_{k_{n},i}))
 \]
for notational simplicity.  
After renumbering, Lemma~\ref{scalingballs} implies that there exists $M=M(R)$ such that
 \[
 D_R \subset \bigcup_{i=1}^M \tilde{B}_{n,i}.
 \]
If we write
 \[
I^l_{k_n}(R)=\{ i \in I^l_{k_n}: i \leq M\} \ \ \ \forall \, l =1, \dots, \Lambda,
\] 
then
 \[
 D_R \subset \bigcup_{l=1}^{\Lambda} \bigcup_{i \in I^l_{k_n}(R)} \tilde{B}_{n,i}.
 \]
Choose a subsequence of  (\ref{subsequence}), which we will denote again by $\{n_j\}$, such that 
\[
\Psi_{n_j}^{-1} \circ \pi_{n_j}^{-1}(x_{k_{n_j},i}) \rightarrow \tilde{x}_i \ \ \ \forall \, i \in \{1, \dots, M\}
\]
and such that for each $l =1, \dots, \Lambda$, the sets 
\[
\tilde{I}^l:=I^l_{k_{n_j}}(R)=\{ i \in I^l_{k_{n_j}}: i \leq M\}
\] 
are equal for all $k_{n_j}$.   Unlike {\bf CASE 1},  where $B_{r_{n_j}}(x_{k_{n_j},i})$ is the same ball $B_i$ for all $j$,  the sets $\tilde{B}_{n_1,i}, \tilde{B}_{n_2,i}, \dots$ are not necessarily the same.
Since the component functions of the pullback metric
$
\Psi_{n_j}^*g
$
converge uniformly to those of the standard Euclidean metric $g_0$ on $\C$ by (\ref{deltaij}) and  
$\tilde{B}_{n_j,i}$ with respect to  $\Psi_{n_j}^*g$ is a ball of radius 1,  $\tilde{B}_{n_j,i}$ with respect to $g_0$ is close to being a ball of radius 1 in the following sense:  for all $\epsilon>0$, there exists $J$ large enough such that for all $j \geq J$, $B_{1-\epsilon}(\tilde{x}_i) \subset \tilde{B}_{n_j,i}$ for $i=1, \ldots, M$. Choose $\epsilon>0$ sufficiently small such that $D_R \subset \bigcup_{i=1}^M B_{1-\epsilon}(\tilde{x}_i)$.  Then choose $J$ as above. Set
\[
      \tilde{B}_i:= \bigcap_{j \geq J} \tilde{B}_{n_j,i} \supset B_{1-\epsilon}(\tilde{x}_i)
      \ \
      \mbox{ and }
      \ \
      t\tilde{B}_i:= \bigcap_{j \geq J} t\tilde{B}_{n_j,i} \mbox{ for } t \in \mathbb{R}^+.
\]
Then
\begin{equation} \label{firstone}
 D_R \subset \bigcup_{i=1}^M \tilde{B}_i =\bigcup_{l=1}^{\Lambda} \bigcup_{i \in \tilde{I}^l(R)} \tilde{B}_i.
\end{equation}
Using (\ref{firstone}), we can now follow {\bf CASE 1} (cf.~(\ref{unc0})) to prove that for $l \in \{1, \dots, \Lambda\}$, 
\begin{equation} \label{unc}
\tilde{u}_n^{l} \rightrightarrows \tilde{u}_R \mbox{ on } \bigcup_{\lambda=1}^{l} \bigcup_{i \in \tilde{I}^{\lambda}}\tilde{B}_i \mbox{ where $\tilde{u}_R$ is as in } (\ref{limitmapCASE2}).
\end{equation}
Let $x \in D_R$.   There exists $l \in \{1, \dots, \Lambda\}$ and $i \in \tilde{I}^l$ such that $x \in \tilde{B}_i$ by (\ref{firstone}).  Since harmonicity is invariant under conformal transformations of the domain, $\tilde{u}^l_{n_j}$ is a energy minimizing on $2\tilde{B}_{n_j,i}$.   Since  $\tilde{B}_i \subset \tilde{B}_{n_j,i} \subset 2\tilde{B}_{n_j,i}$ and  $\tilde{u}^{l}_{n_j} \rightrightarrows \tilde{u}_R$ on $\tilde{B}_i$ by (\ref{unc}),  Lemma~\ref{luckhaus} implies that $\tilde{u}_R$ is energy minimizing on $\frac{1}{2}\tilde{B}_i$. Since $x$ is an arbitrary point in $D_R$, we have shown that $\tilde{u}_R$ is harmonic on $D_R$.  

Finally, by the conformal invariance of energy,  $E(\tilde{u}_n^l) = E(u_n^l\big|_{B_1(y_n)}) \leq E(u_0^0)$.  By the lower semicontinuity of energy and (\ref{energybound}), we have \begin{equation} \label{energyofu_R}
E(\tilde u_R) \leq E(u_0^0).
\end{equation}

By considering a  compact exhaustion $\{D_{2^m}\}_{m=1}^{\infty}$ of $\C$ and a diagonalization procedure, we  prove the existence of a harmonic map $\tilde{u}:\C \rightarrow X$.  By (\ref{energyofu_R}), 
\[
E(\tilde{u}) \leq E(u_0^0).
\] 
It follows from (\ref{nonconst1}) and (\ref{nonconst2}) that $\tilde{u}$ is nonconstant.
Thus, {\bf CASE 2} is complete by applying the removable singularity result Corollary \ref{removable-singularity}.

\appendix

\section{Quadrilateral Estimates} \label{section:quadrilateral-estimates}

In this section, we include several estimates for quadrilaterals in a CAT(1) space. The estimates are stated in the unpublished thesis \cite{serbinowski} without proof. As the calculations were not obvious, we include our proofs for the convenience of the reader.  References to the location of each estimate in \cite{serbinowski} are also included.

The first lemma is a result of Reshetnyak which will be essential in later estimates. 
\begin{lemma}[{\cite[Lemma 2]{Reshet}}]
Let $\Box PQRS$ be a quadrilateral in $X$.  Then the sum of the length of diagonals in $\Box PQRS$ can be estimated as follows:
\begin{align} \label{equation:diagonal_sphere}
\begin{split}
	\cos d_{PR} + \cos d_{QS} &\ge - \frac{1}{2} (d_{PQ}^2 + d_{RS}^2) + \frac{1}{4} (1+\cos d_{PS}) (d_{QR} - d_{PS})^2\\
	&\quad  +\cos d_{QR} + \cos d_{PS}  +\mathrm{Cub}\left( d_{PQ}, d_{RS}, d_{QR}-d_{SP}\right).
\end{split}
\end{align}
\end{lemma}
\begin{proof}
It suffices to prove the inequality holds for a quadrilateral  $\Box PQRS$ in $\mathbb S^2$. By viewing $\mathbb S^2$ as a unit sphere in  $\mathbb{R}^3$, the points $P, Q, R, S$ determine a quadrilateral in $\mathbb{R}^3$. Applying the identity for the quadrilateral in $\mathbb{R}^3$ (cf. \cite[Corollary 2.1.3]{korevaar-schoen1}), 
\begin{align*}
	\overline{PR}^2 + \overline{QS}^2 &\leq \overline{PQ}^2 + \overline{QR}^2 + \overline{RS}^2 + \overline{SP}^2 -( \overline{SP} - \overline{QR})^2
\end{align*}
where $\overline{AB}$ denotes the Euclidean distance between $A$ and $B$ in $\mathbb{R}^3$. To prove this, consider the vectors $A=Q-P, B=R-Q, C=S-R, D=P-S$. Then
\begin{align*}
\overline{PR}^2 + \overline{QS}^2 &= \frac 12 \left( |A+B|^2 + |C+D|^2 +|B+C|^2+|D+A|^2\right)\\
&= |A|^2+|B|^2+ |C|^2+|D|^2+\left( A \cdot B + C \cdot B + D \cdot A +D\cdot C\right)\\
& = |A|^2+|B|^2+ |C|^2+|D|^2 - |B+D|^2 \text{ since } A+B+C+D=0\\
& \leq  |A|^2+|B|^2+ |C|^2+|D|^2 - \left||B| - |D| \right|^2. 
\end{align*}

Note that $\overline{AB}^2 = 2-2\cos d_{AB}$, we obtain 
\begin{align*}
	\cos d_{PR} + \cos d_{QS} &= - 2 + \cos d_{PQ} + \cos d_{RS} + \cos d_{QR} + \cos d_{PS}\\
	&\quad +\frac{1}{2} \left(\sqrt{2-2\cos d_{QR}} - \sqrt{2-2\cos d_{SP}} \right)^2.
\end{align*}
The lemma follows from the following Taylor expansion:
\begin{align*}
	- 2 + \cos d_{PQ} + \cos d_{RS} &= - \frac{1}{2} d_{PQ}^2 -\frac{1}{2} d_{RS}^2 + O( d_{RS}^4 + d_{PQ} ^4)\\
	\left(\sqrt{2-2\cos d_{QR}} - \sqrt{2-2\cos d_{SP}} \right)^2&=\left( \frac{ \sin d_{SP}}{\sqrt{2-2\cos d_{SP}}}(d_{QR} - d_{SP}) + O\left((d_{QR} - d_{SP})^2\right)\right)^2\\
	&= \frac{1+\cos d_{PS}}{2}(d_{QR} - d_{SP})^2 + O\left((d_{QR} - d_{SP})^3\right).
\end{align*}
\end{proof}

\begin{lemma}[{\cite[Estimate I, Page 11]{serbinowski}}] \label{lemma:estimateI}
Let $\Box PQRS$ be a quadrilateral in the  CAT(1) space $X$.  Let $P_{\frac{1}{2}}$ be the mid-point between $P$ and $S$, and let $Q_{\frac{1}{2}}$ be the mid-point between $Q$ and $R$.  Then 
\begin{align*}
	\cos^2 \left( \frac{d_{PS}}{2} \right) d^2(Q_{\frac{1}{2}}, P_{\frac{1}{2}}) &\le \frac{1}{2}  (d_{PQ}^2 + d_{RS}^2) -\frac{1}{4}  (d_{QR} - d_{PS})^2 \\&\quad
	+ \mathrm{Cub}\left(d_{PQ}, d_{RS}, d(P_{\frac 12}, Q_{\frac 12}), d_{QR}-d_{SP} \right).
\end{align*} 
\end{lemma}
\begin{proof}
As a direct consequence of law of cosine (see also the figure below),   we have the following inequalities
\begin{align*}
	\cos d(Q_{\frac{1}{2}}, P_{\frac{1}{2}}) &\ge \alpha \left(\cos d(Q_{\frac{1}{2}}, S) + \cos d (Q_{\frac{1}{2}}, P) \right)\\
	\cos d(Q_{\frac{1}{2}}, S) & \ge \beta  \left(\cos d_{R S} + \cos d_{QS} \right)\\
	\cos d(Q_{\frac{1}{2}}, P) & \ge \beta  \left(\cos d_{RP} + \cos d_{QP} \right)
\end{align*}
where 
\begin{align*}
\alpha =\frac{1}{2\cos \left(\frac{d_{PS}}{2}\right)}\quad \mbox{and} \quad \beta = \frac{1}{2\cos \left( \frac{d_{QR}}{2}\right)}.
\end{align*} 
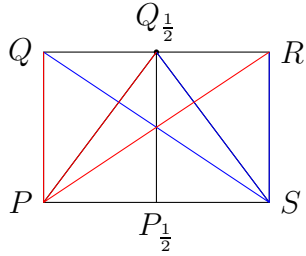
\begin{center}
\begin{tikzpicture}
\draw (0,0) node[anchor = east]{$P$} -- (3,0) node[anchor =west] {$S$} -- (3,2) node [anchor = west]{$R$}-- (0, 2) node[anchor= east]{$Q$}-- cycle;
\draw (1.5, 2) node[anchor = south]{$Q_{\frac{1}{2}}$} --(1.5,0)  node[anchor =north]  {$P_{\frac{1}{2}}$};
\fill (1.5,2) circle (1pt);
\draw (0,0) -- (1.5, 2);
\draw (3,0) -- (1.5,2);
\draw[blue] (3,0) -- (0,2);
\draw[blue] (1.5, 2) -- (3 ,0);
\draw[blue] (3,0) -- (3,2);
\draw [red] (3,2) -- (0,0);
\draw[red] (1.5,2) -- (0,0); 
\draw[red] (0,2) -- (0,0);
\end{tikzpicture}
\end{center}
Combining the above inequalities yields 
\begin{align*}
		\cos d(Q_{\frac{1}{2}}, P_{\frac{1}{2}})  \ge \alpha\beta \left(\cos d_{RS} + \cos d_{QS} + \cos d_{R P} + \cos d_{Q P} \right).
\end{align*}
We apply \eqref{equation:diagonal_sphere} for the sum of diagonals $\cos d_{QS} + \cos d_{R P}$ and Taylor expansion for $\cos d_{RS} $ and $\cos d_{QP} $. It yields 
\begin{align*}
	\cos d(Q_{\frac{1}{2}}, P_{\frac{1}{2}})& \ge \alpha \beta \bigg(2-(d_{PQ}^2 + d_{RS}^2) + \frac{1}{4} (1+ \cos d_{PS}) (d_{QR} - d_{PS})^2 + \cos d_{QR} + \cos d_{PS}\bigg)\\ & \quad +\mathrm{Cub}\left( d_{PQ}, d_{RS}, d_{QR}-d_{SP}\right)\\
	&= \alpha\beta\left(2+  \cos d_{QR} + \cos d_{PS} + \frac{1}{4} (1+ \cos d_{PS}) (d_{QR} - d_{PS})^2\right)- \alpha\beta (d_{PQ}^2 + d_{RS}^2) \\
	&\quad +   \mathrm{Cub}\left( d_{PQ}, d_{RS}, d_{QR}-d_{SP}\right).
\end{align*}
Note that 
\begin{align*}
	&2+  \cos d_{QR} + \cos d_{PS} + \frac{1}{4} (1+ \cos d_{PS}) (d_{QR} - d_{PS})^2 \\
	&=2( \cos^2\frac{d_{QR}}{2} + \cos^2 \frac{d_{PS}}{2}) + \frac{1}{2}\cos^2 \frac{d_{PS}}2(d_{QR} - d_{PS})^2\\
	&= 2\left(\cos \frac{d_{QR}}{2} - \cos \frac{d_{PS}}{2} \right)^2 + 4 \cos \frac{d_{QR}}{2}\cos \frac{d_{PS}}{2} +\frac{1}{2}\cos^2 \frac{d_{PS}}2 (d_{QR} - d_{PS})^2\\
	&= \frac{1}{2} \sin^2 \frac{d_{PS}}{2} (d_{QR}-d_{PS})^2 + 4 \cos \frac{d_{QR}}{2}\cos \frac{d_{PS}}{2} +\frac{1}{2}\cos^2 \frac{d_{PS}}2(d_{QR} - d_{PS})^2 + O(|d_{QR}-d_{PS}|^3)\\
	&= \frac{1}{2} (d_{QR}-d_{PS})^2 + 4 \cos \frac{d_{QR}}{2}\cos \frac{d_{PS}}{2}+ O(|d_{QR}-d_{PS}|^3).
\end{align*}
Since $\alpha\beta = \alpha^2 + O(|d_{QR} - d_{PS}|)$, we have
\begin{align*}
	\cos d(Q_{\frac{1}{2}}, P_{\frac{1}{2}})& \ge 1 -\alpha^2 (d_{PQ}^2 + d_{RS}^2)+ \frac{1}{2} \alpha^2 (d_{QR}-d_{PS})^2  +\mathrm{Cub}\left(d_{PQ},d_{RS},d_{QR}-d_{SP}\right).\end{align*}
The lemma follows as 
\[
	\cos d(Q_{\frac{1}{2}}, P_{\frac{1}{2}}) = 1 - \frac{d^2(Q_{\frac{1}{2}}, P_{\frac{1}{2}} )}{2}  + O(d^4(Q_{\frac{1}{2}}, P_{\frac{1}{2}} ) ).
\]
\end{proof}

\begin{definition}
Given a metric space $(X,d)$ and a geodesic $\gamma_{PQ}$ with $d_{PQ}<\pi$, for $\tau \in [0,1]$ let $(1-\tau) P + \tau Q$ denote the point on $\gamma_{PQ}$ at distance $\tau d_{PQ}$ from $P$. That is
\[
d((1-\tau )P + \tau Q,P) = \tau d_{PQ}.
\]
\end{definition}

\begin{lemma}[cf. {\cite[Estimate II, Page 13]{serbinowski}}] \label{lemma:estimateII}
Let $\Delta PQS$ be a triangle in the CAT(1) space $X$. For a  pair of numbers $0\le \eta, \eta' \le 1$ define 
\begin{align*}
	P_{\eta'} &= (1-\eta') P + \eta' Q\\
	S_\eta& = (1-\eta) S + \eta Q. 
\end{align*}
Then 
\begin{align*}
	d^2 (P_{\eta'}, S_{\eta}) &\le  \frac{\sin^2((1-\eta) d_{QS})}{\sin^2 d_{QS}} (d_{PS}^2- (d_{QS}-d_{QP})^2) +  \left((1-\eta)(d_{QS}-d_{QP}) + (\eta' - \eta) d_{QS}\right)^2  \\
	&\quad+ \mathrm{Cub}\left( d_{PS}, d_{QS}-d_{QP}, \eta-\eta' \right).
\end{align*}
\end{lemma}
\begin{proof}Again we prove the inequality for a quadrilateral on $\mathbb S^2$. 
Denote $x = d_{QS}$ and $y = d_{QP}$. Denote 
\[
		\alpha_\eta= \frac{\sin(\eta d_{QS})}{\sin d_{QS}} =\frac{\sin(\eta x)}{\sin x}, \qquad \beta_{\eta'} = \frac{\sin(\eta' d_{QP})}{\sin d_{QP}} = \frac{\sin(\eta' y)}{ \sin y}.
\]
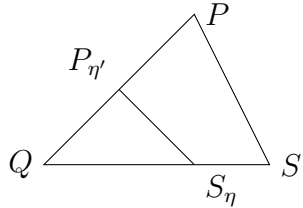
\begin{center}
\begin{tikzpicture}
\draw (0,0) node[anchor = east]{$Q$} --  (3, 0) node[anchor =west] {$S$} -- (2,2) node[anchor = west]{$P$}  -- cycle;
\draw (2,0) node[anchor=north west]{$S_\eta$}-- (1,1) node[anchor=south east]{$P_{\eta'}$};
\end{tikzpicture}
\end{center}
By the law of cosines on the sphere (see the figure above),
\begin{align*}
	\cos d_{PS} &= \cos x \cos y+ \sin x \sin y \cos \theta = \cos (x-y) + \sin x \sin y (\cos \theta -1)\\
	\cos d({P_{\eta'}}, S_{\eta}) &\ge  \cos ((1-\eta) x) \cos ((1-\eta')y)+ \sin ((1-\eta)x) \sin ((1-\eta')y )\cos \theta\\
	& = \cos ((1-\eta)x- (1-\eta')y) + \sin ((1-\eta)x) \sin ((1-\eta')y) (\cos \theta -1),
\end{align*}
where $\theta$ denotes the angle $\angle PQS$  on $\mathbb S^2$. Substituting the term  $(\cos\theta -1)$ of  the second inequality with the one in the first identity, we obtain
\begin{align*}
	\cos d({P_{\eta'}}, S_{\eta})  &\ge  \cos ((1-\eta)x- (1-\eta')y) + \alpha_{1-\eta} \beta_{1-\eta'} (\cos d_{PS} - \cos (x-y))\\
	&= \cos \left((1-\eta) (x-y) + (\eta' - \eta)x + (\eta'-\eta)(y-x)\right) + \alpha_{1-\eta}^2 (\cos d_{PS} - \cos (x-y))\\
	&\quad + \alpha_{1-\eta} (\beta_{1-\eta'}-\alpha_{1-\eta})(\cos d_{PS} - \cos (x-y)).
\end{align*}
Using the Taylor expansion $\cos a = 1 - \frac{a^2}{2} +O(a^4)$ and $(\beta_{1-\eta'}-\alpha_{1-\eta}) = O(|\eta'-\eta|+ |x-y| )$, we derive 
\begin{align*}
	\cos d({P_{\eta'}}, S_{\eta})  &\ge 1 - \frac{((1-\eta)(x-y) + (\eta' - \eta) x)^2}{2} + \alpha_{1-\eta}^2 \left( - \frac{d_{PS}^2}{2} + \frac{(x-y)^2}{2} \right)\\
	&\quad + \mathrm{Cub}\left(|\eta'-\eta| ,|x-y|,  d_{PS}\right).
\end{align*}
It implies that 
\begin{align*}
	d^2({P_{\eta'}}, S_{\eta}) &\le \alpha^2_{1-\eta} (d_{PS}^2- (x-y)^2) +  ((1-\eta)(x-y) + (\eta' - \eta) x)^2  \\
		&\quad +\mathrm{Cub}\left(|\eta'-\eta| ,|x-y|,  d_{PS}\right).
\end{align*}
\end{proof}

\begin{corollary}\label{corollary:interpolation-point}
 Let $u: \Omega \to \mathcal B_{\rho}(Q)$ be a finite energy map and $\eta\in C^\infty_C(\Omega, [0,1])$. Define $\hat u : \Omega \to \mathcal B_{\rho}(Q)$ as
\[
	\hat{u}(x) = (1-\eta(x)) u(x) + \eta(x)Q. 
\]
Then $\hat{u}$ has finite energy, and for any smooth vector field $W \in \Gamma(\Omega)$ we have
\[
	|(\hat{u})_*(W)|^2 \le \left(\frac{\sin(1-\eta)R^u}{\sin R^u} \right)^2 (|u_*(W)|^2 - |\nabla_W R^u|^2) + |\nabla_W ((1-\eta)R^u)|^2,
\]
where $R^u(x) = d(u(x), Q)$.
\end{corollary}
Note that every error term that appeared in Lemma \ref{lemma:estimateII} will converge to the product of an $L^1$ function and a term that goes to zero. So all error terms vanish when taking limits.

\begin{lemma}[cf. {\cite[Estimate III, page 19]{serbinowski}}]\label{EstIII}
Let $\Box PQRS$ be a quadrilateral in a CAT(1) space $X$. For $\eta', \eta \in [0,1]$ define
\[
	Q_{\eta'} = (1-\eta') Q + \eta' R, \quad P_{\eta} = (1-\eta) P + \eta S. 
\]
Then 
\begin{align*}
	d^2(Q_{\eta'}, P_{\eta})&	 + d^2 (Q_{1-\eta'}, P_{1-\eta})\\
	&\le \left(1 + 2\eta d_{PS} \tan (\frac{1}{2} d_{PS}) \right) (d_{PQ}^2 + d_{RS}^2) - 2\eta \left( 1+ \frac{1}{2} d_{PS} \tan( \frac{1}{2} d_{PS})  \right) (d_{QR} - d_{PS})^2 \\
	&\quad +2(2\eta-1)(\eta' - \eta) d_{PS} (d_{QR} - d_{PS})\\
	&\quad + \eta^2 \mathrm{Quad}(d_{PQ}, d_{RS}, d_{QR}-d_{PS})+ \mathrm{Cub}\left(d_{QR} - d_{PS}, d_{PQ}, d_{RS},\eta-\eta' \right)\\
	%
\end{align*}
\end{lemma}
\begin{proof}
For notation simplicity, we denote 
\[
	x= d_{PS}, \qquad y = d_{QR},\qquad	\alpha_\eta = \frac{\sin (\eta x)}{\sin x},\qquad \beta_{\eta'} = \frac{\sin (\eta' y)}{\sin y}.
\]
Apply  \cite[Definition 1.6]{serbinowski}  to each of the blue, red, and yellow triangles below.
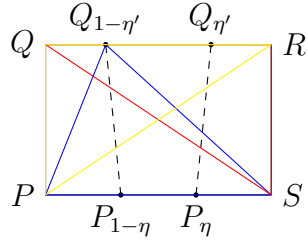
\begin{center}
\begin{tikzpicture}
\draw (0,0) node[anchor = east]{$P$} -- (3,0) node[anchor =west] {$S$} -- (3,2) node [anchor = west]{$R$}-- (0, 2) node[anchor= east]{$Q$}-- cycle;
\draw[dashed] (1,0) node[anchor = north ]{$P_{1-\eta}$}-- (0.8,2) node[anchor=south]{$Q_{1-\eta'}$};
\draw[dashed] (2,0) node[anchor = north ]{$P_{\eta}$}-- (2.2,2) node[anchor=south]{$Q_{\eta'}$};
\fill (1,0) circle (1pt);
\fill (0.8,2) circle (1pt);
\fill (2,0) circle (1pt);
\fill (2.2,2) circle (1pt);
\draw[blue] (0.8,2) -- (0,0);
\draw[blue] (0.8, 2) -- (3 ,0);
\draw[blue] (0,0) -- (3,0);
\draw[red] (3,0) -- (0,2);
\draw[red] (3,0) -- (3, 2);
\draw[red] (0, 2)--(3, 2);
\draw[yellow] (0,0) -- (0,2); 
\draw[yellow] (0,0) -- (3,2);
\draw[yellow] (3,2)--(0,2);
\end{tikzpicture}
\end{center}
We derive
\begin{align*}
	\cos d(Q_{1-\eta'}, P_{1-\eta}) &\ge \alpha_{\eta} \cos d(Q_{1-\eta'}, S) + \alpha_{1-\eta} \cos d(Q_{1-\eta'}, P)\\
	&\ge \alpha_{\eta} (\beta_{\eta'}\cos d_{SR} + \beta_{1-\eta'} \cos d_{SQ}) + \alpha_{1-\eta} (\beta_{\eta'} \cos d_{PR} + \beta_{1-\eta'} \cos d_{PQ}).
\end{align*}
Compute similarly for $d(Q_{\eta'}, P_{\eta})$ for the highlighted triangles below:
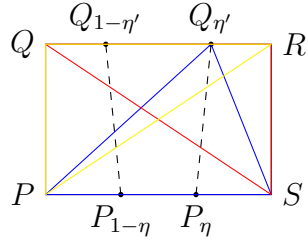
\begin{center}
\begin{tikzpicture}
\draw (0,0) node[anchor = east]{$P$} -- (3,0) node[anchor =west] {$S$} -- (3,2) node [anchor = west]{$R$}-- (0, 2) node[anchor= east]{$Q$}-- cycle;
\draw[dashed] (1,0) node[anchor = north ]{$P_{1-\eta}$}-- (0.8,2) node[anchor=south]{$Q_{1-\eta'}$};
\draw[dashed] (2,0) node[anchor = north ]{$P_{\eta}$}-- (2.2,2) node[anchor=south]{$Q_{\eta'}$};
\fill (1,0) circle (1pt);
\fill (0.8,2) circle (1pt);
\fill (2,0) circle (1pt);
\fill (2.2,2) circle (1pt);
\draw[blue] (2.2,2) -- (0,0);
\draw[blue] (2.2, 2) -- (3 ,0);
\draw[blue] (0,0) -- (3,0);
\draw[red] (3,0) -- (0,2);
\draw[red] (3,0) -- (3, 2);
\draw[red] (0, 2)--(3, 2);
\draw[yellow] (0,0) -- (0,2); 
\draw[yellow] (0,0) -- (3,2);
\draw[yellow] (3,2)--(0,2);
\end{tikzpicture}
\end{center}
We derive
\begin{align*}
	\cos d(Q_{\eta'}, P_{\eta}) &\ge \alpha_{\eta} \cos d(Q_{\eta'}, P) + \alpha_{1-\eta} \cos d(Q_{\eta'}, S)\\
	&\ge \alpha_{\eta} (\beta_{\eta'} \cos d_{PQ} + \beta_{1-\eta'} \cos d_{PR}) + \alpha_{1-\eta} (\beta_{\eta'} \cos d_{SQ} + \beta_{1-\eta'} \cos d_{SR}).
\end{align*}
Adding the above two inequalities, we obtain
\begin{align} \label{equation:quadrilateral-sphere}
\begin{split}
&\cos d(Q_{1-\eta'}, P_{1-\eta}) +\cos d(Q_{\eta'}, P_{\eta})  \\
	&\ge (\alpha_\eta \beta_{\eta'} + \alpha_{1-\eta} \beta_{1-\eta'}) (\cos d_{PQ} + \cos d_{SR}) + (\alpha_\eta \beta_{1-\eta'} +\alpha_{1-\eta} \beta_{\eta'}) (\cos d_{PR} + \cos d_{SQ}).
\end{split}
\end{align}

Applying \eqref{equation:diagonal_sphere} to the term $\cos d_{PR} + \cos d_{SQ}$ and using Taylor expansion, 
the inequality~\eqref{equation:quadrilateral-sphere} becomes
\begin{align*}
	&\cos d(Q_{1-\eta'}, P_{1-\eta}) +\cos d(Q_{\eta'}, P_{\eta}) \notag  \ge  (\alpha_\eta \beta_{\eta'} + \alpha_{1-\eta} \beta_{1-\eta'}) \left(2-\frac{d_{PQ}^2}{2} - \frac{d_{SR}^2}{2}\right) \notag \\
	&\quad +  (\alpha_\eta \beta_{1-\eta'} +\alpha_{1-\eta} \beta_{\eta'}) \left( -\frac{1}{2} (d_{PQ}^2 + d_{SR}^2) + \frac{1}{4} (1+\cos d_{PS}) (d_{QR} - d_{PS})^2 + \cos d_{QR} + \cos d_{PS}\right)\notag\\
	&\quad + \mathrm{Cub}\left( d_{PQ}, d_{RS}, d_{QR}-d_{SP}\right).\notag\\
\end{align*}Hence, 
\begin{align}
	\cos d(Q_{1-\eta'}, P_{1-\eta}) &+\cos d(Q_{\eta'}, P_{\eta}) \notag \\
	&\ge -\frac{1}{2}  (\alpha_\eta \beta_{\eta'} + \alpha_{1-\eta} \beta_{1-\eta'} + \alpha_\eta \beta_{1-\eta'} +\alpha_{1-\eta} \beta_{\eta'})(d_{PQ}^2 + d_{SR}^2) \label{equation:1}\\
	&\quad + 2(\alpha_\eta \beta_{\eta'} + \alpha_{1-\eta} \beta_{1-\eta'}) +  (\alpha_\eta \beta_{1-\eta'} +\alpha_{1-\eta} \beta_{\eta'}) ( \cos d_{QR} + \cos d_{PS})\label{equation:2}\\
	&\quad  +  \frac{1}{4}  (\alpha_\eta \beta_{1-\eta'} +\alpha_{1-\eta} \beta_{\eta'})(1+\cos d_{PS}) (d_{QR} - d_{PS})^2 \label{equation:3}\\
	&\quad  +\mathrm{Cub}\left( d_{PQ}, d_{RS}, d_{QR}-d_{SP}\right).  \notag 
\end{align}

We need the following elementary trigonometric identities to compute \eqref{equation:1}, \eqref{equation:2}, \eqref{equation:3}:
\begin{align*}
	\alpha_\eta \beta_{\eta'} + \alpha_{1-\eta} \beta_{1-\eta'} &= \frac{\sin (\eta - \frac{1}{2})x \sin (\eta' -\frac{1}{2})y}{2 \sin \frac{1}{2} x \sin \frac{1}{2} y} + \frac{\cos (\eta - \frac{1}{2})x \cos (\eta' -\frac{1}{2})y}{2 \cos \frac{1}{2} x \cos \frac{1}{2} y} \\
	\alpha_\eta \beta_{1-\eta'} + \alpha_{1-\eta} \beta_{\eta'} &= - \frac{\sin (\eta - \frac{1}{2})x \sin (\eta' -\frac{1}{2})y}{2 \sin \frac{1}{2} x \sin \frac{1}{2} y} + \frac{\cos (\eta - \frac{1}{2})x \cos (\eta' -\frac{1}{2})y}{2 \cos \frac{1}{2} x \cos \frac{1}{2} y} \\
	\left( \frac{\cos(\eta-\frac{1}{2})x}{\cos \frac{1}{2}x} \right)^2 &= 1 + 2\eta x \tan\frac{1}{2} x + O(\eta^2).
\end{align*}
Noting that
\begin{align*}
	 &\alpha_\eta \beta_{\eta'} + \alpha_{1-\eta} \beta_{1-\eta'} + \alpha_\eta \beta_{1-\eta'} +\alpha_{1-\eta} \beta_{\eta'}= \frac{\cos (\eta - \frac{1}{2})x \cos (\eta' -\frac{1}{2})y}{\cos \frac{1}{2} x \cos \frac{1}{2} y}\\
	 &=\left( \frac{\cos(\eta-\frac{1}{2})x}{\cos \frac{1}{2}x} \right)^2+  O(|\eta - \eta'|+ |x-y| )\\
	 &= 1 + 2 \eta x \tan (\frac{1}{2} x)+  O(\eta^2+ |\eta - \eta'|+|x-y|),
\end{align*}

we obtain for \eqref{equation:1}
\begin{align*}
	&-\frac{1}{2}  (\alpha_\eta \beta_{\eta'} + \alpha_{1-\eta} \beta_{1-\eta'} + \alpha_\eta \beta_{1-\eta'} +\alpha_{1-\eta} \beta_{\eta'})(d_{PQ}^2 + d_{SR}^2)\\
	&=-\frac{1}{2} \left(1 + 2 \eta x \tan (\frac{1}{2} x)\right)(d_{PQ}^2 + d_{SR}^2)+
	O\left((\eta^2+ |\eta - \eta'|+|x-y|) (d_{PQ}^2 + d_{SR}^2)\right).
\end{align*}

\begin{lemma}
We can compute \eqref{equation:2} as follows:
\begin{align*}
	&2(\alpha_\eta \beta_{\eta'} + \alpha_{1-\eta} \beta_{1-\eta'}) +  (\alpha_\eta \beta_{1-\eta'} +\alpha_{1-\eta} \beta_{\eta'}) ( \cos x + \cos y)\\
	&=2-  \left((\eta- \frac{1}{2})(y-x) + (\eta' - \eta)x\right)^2  + \frac{\sin^2 (\eta - \frac{1}{2})x}{4\sin^2 \frac{1}{2} x} \cos^2 (\frac{1}{2} x) (x-y)^2 \\ & \quad+ \frac{\cos^2 (\eta- \frac{1}{2})x}{4\cos^2 \frac{1}{2} x} \sin^2 (\frac{1}{2}x) (x-y)^2+ O(|x-y|^2(|x-y| +|\eta'-\eta|)).
\end{align*}
\end{lemma}
\begin{proof}
\begin{align*}
&2(\alpha_\eta \beta_{\eta'} + \alpha_{1-\eta} \beta_{1-\eta'}) +  (\alpha_\eta \beta_{1-\eta'} +\alpha_{1-\eta} \beta_{\eta'}) ( \cos x + \cos y)\\
& =  \frac{\sin (\eta - \frac{1}{2})x \sin (\eta' -\frac{1}{2})y}{2 \sin \frac{1}{2} x \sin \frac{1}{2} y} (2- \cos x- \cos y) +\frac{\cos (\eta - \frac{1}{2})x \cos (\eta' -\frac{1}{2})y}{2 \cos \frac{1}{2} x \cos \frac{1}{2} y}(2+ \cos x + \cos y).
\end{align*}
 Note that 
\begin{align*}	2-\cos x - \cos y&= 2(\sin \frac{1}{2} x)^2 + 2(\sin \frac{1}{2} y)^2 = 2\left(2 \sin \frac{1}{2} x \sin \frac{1}{2} y + (\sin\frac{1}{2} x - \sin \frac{1}{2} y)^2\right)\\
&=4 \sin \frac{1}{2} x \sin \frac{1}{2} y  + \frac{1}{2} (\cos \frac{1}{2} x)^2 (x-y)^2 + O(|x-y|^3)\\
	2+\cos x + \cos y &= 2(\cos \frac{1}{2} x)^2 + 2 (\cos \frac{1}{2} y)^2 = 2\left(2 \cos \frac{1}{2} x \cos \frac{1}{2} y +  (\cos\frac{1}{2} x - \cos \frac{1}{2} y)^2\right)\\
	&= 4  \cos \frac{1}{2} x \cos \frac{1}{2} y +\frac 12 (\sin \frac{1}{2} x)^2 (x-y)^2 + O(|x-y|^3),
\end{align*}
where we apply Taylor expansion in the last equality. Hence we have
\begin{align*}
&2(\alpha_\eta \beta_{\eta'} + \alpha_{1-\eta} \beta_{1-\eta'}) +  (\alpha_\eta \beta_{1-\eta'} +\alpha_{1-\eta} \beta_{\eta'}) ( \cos x + \cos y)\\
&=2 \left(\sin (\eta - \frac{1}{2} )x \sin (\eta' - \frac{1}{2})y + \cos (\eta - \frac{1}{2} )x \cos (\eta' -\frac{1}{2})y\right) + \frac{\sin^2 (\eta - \frac{1}{2})x}{4\sin^2 \frac{1}{2} x} (\cos \frac{1}{2} x)^2 (x-y)^2 \\& \quad + \frac{\cos^2 (\eta- \frac{1}{2})x}{4\cos^2 \frac{1}{2} x} (\sin \frac{1}{2}x)^2 (x-y)^2  + O(|x-y|^2(|x-y| +|\eta'-\eta|)).
\end{align*} Here we use the estimates
\[
 \frac{\sin (\eta - \frac{1}{2})x \sin (\eta' -\frac{1}{2})y}{2 \sin \frac{1}{2} x \sin \frac{1}{2} y} -  \frac{\sin^2 (\eta - \frac{1}{2})x}{2 \sin^2 \frac{1}{2} x }=  O(|\eta - \eta'|+ |x-y| )
\] and 
\[
\frac{\cos (\eta - \frac{1}{2})x \cos (\eta' -\frac{1}{2})y}{2 \cos \frac{1}{2} x \cos \frac{1}{2} y}-\frac{\cos^2 (\eta - \frac{1}{2})x }{2 \cos^2 \frac{1}{2} x}= O(|\eta - \eta'|+ |x-y| ).
\] 

Observe that 
\begin{align*}
	\left(\sin (\eta - \frac{1}{2} )x \sin (\eta' - \frac{1}{2})y \right. &\left. + \cos (\eta - \frac{1}{2} )x \cos (\eta' -\frac{1}{2})y\right) \\
	&=  \cos \left( (\eta- \frac{1}{2})(y-x) + (\eta' - \eta)x + (\eta' -\eta) (y-x) \right)
\end{align*}
and use $\cos a = 1 - \frac{a^2}{2} + O(a^4)$.
\end{proof}
\begin{lemma} 
Adding the terms in the previous computational lemma that contain $(x-y)^2$ to \eqref{equation:3}, we have the following estimate:
\begin{align*}
	&\frac{1}{4}  (\alpha_\eta \beta_{1-\eta'} +\alpha_{1-\eta} \beta_{\eta'})(1+\cos x) (x-y)^2  \\
	& - (\eta - \frac{1}{2})^2  (x-y)^2 + \frac{\sin^2 (\eta - \frac{1}{2})x}{4\sin^2 \frac{1}{2} x} \cos^2 (\frac{1}{2} x) (x-y)^2 + \frac{\cos^2 (\eta- \frac{1}{2})x}{4\cos^2 \frac{1}{2} x} \sin^2 (\frac{1}{2}x) (x-y)^2\\
	& =  \eta(1+\frac{1}{2} x \tan \frac{1}{2} x) (x-y)^2 + O( |x-y|^2(\eta^2+|x-y|+|\eta-\eta'|) ).
\end{align*}
\end{lemma}
\begin{proof}
Noting that $1+\cos x = 2\cos^2( \frac{1}{2} x)$, we have that
\begin{align*}
	&\frac{1}{4}  (\alpha_\eta \beta_{1-\eta'} +\alpha_{1-\eta} \beta_{\eta'})(1+\cos x) (x-y)^2  \\
	&= \frac{1}{4} \left( -\left(\frac{\sin (\eta - \frac{1}{2} )x}{\sin \frac{1}{2} x}\right)^2 + \left( \frac{\cos (\eta - \frac{1}{2}) x}{\cos \frac{1}{2} x}\right)^2 \right) \cos^2(\frac{1}{2} x) (x-y)^2 + O(|x-y|^2(|\eta - \eta'|+ |x-y|) ).
\end{align*}
Therefore, 
\begin{align*}
&\frac{1}{4}  (\alpha_\eta \beta_{1-\eta'} +\alpha_{1-\eta} \beta_{\eta'})(1+\cos x) (x-y)^2  \\
	& -(\eta - \frac{1}{2})^2  (x-y)^2 + \frac{\sin^2 (\eta - \frac{1}{2})x}{4\sin^2 \frac{1}{2} x} \cos^2 (\frac{1}{2} x) (x-y)^2 + \frac{\cos^2 (\eta- \frac{1}{2})x}{4\cos^2 \frac{1}{2} x} \sin^2 (\frac{1}{2}x) (x-y)^2\\
	 &=\left( \frac{\cos^2 (\eta- \frac{1}{2})x}{4\cos^2 \frac{1}{2} x} - (\eta - \frac{1}{2})^2 \right) (x-y)^2+O(|x-y|^2(|\eta - \eta'|+ |x-y|) )\\
	 &=\left(\frac{1}{4} + \frac{1}{2} \eta x \tan \frac{1}{2} x -  (-\eta +\frac{1}{4} ) \right) (x-y)^2 +  O(|x-y|^2(\eta^2 + |\eta - \eta'|+ |x-y|) ).
\end{align*}
\end{proof}
Combing the above computations, we have that
\begin{align*}
	\cos d(Q_{1-\eta'}, P_{1-\eta}) +\cos d(Q_{\eta'}, P_{\eta}) 
	&\ge 2-\frac{1}{2} \left(1 + 2 \eta d_{PS} \tan (\frac{1}{2} d_{PS})\right)(d_{PQ}^2 + d_{SR}^2)\\ & \quad +\eta(1+\frac{1}{2} d_{PS} \tan \frac{1}{2} d_{PS}) (d_{QR}-d_{PS})^2\\
	&\quad - (2\eta-1)(\eta' -\eta) d_{PS}(d_{QR}-d_{PS}) \\& \quad + \eta^2 \mathrm{Quad}(d_{PQ}, d_{RS}, d_{QR}-d_{PS})\\
	&\quad+
	 \mathrm{Cub}\left( d_{QR}-d_{PS}, d_{PQ}, d_{RS}, \eta'-\eta \right).\\
	 \end{align*}Taylor expansion gives the result.
\end{proof}

\begin{corollary} \label{corollary:interpolation}
Given a pair of finite energy maps $u_0, u_1 \in W^{1,2}(\Omega, X)$ with images $u_i(\Omega)\subset \mathcal B_\rho (Q)$ and a function $\eta \in C_c^1(\Omega)$, $0\le \eta \le \frac{1}{2}$, define the maps
\begin{align*}
	u_\eta(x) &= (1-\eta(x))u_0 (x) + \eta(x) u_1(x)\\
	u_{1-\eta}(x) &= \eta (x) u_0 (x) + (1-\eta(x)) u_1(x)\\
	d(x)&= d(u_0(x), u_1(x)).
\end{align*}
Then $u_\eta, u_{1-\eta} \in W^{1,2}(\Omega, X)$ and 
\begin{align*} 
\begin{split}
	|\nabla u_\eta|^2 + |\nabla u_{1-\eta}|^2 &\le (1+ 2 \eta d  \tan \frac{d}{2}) (|\nabla u_0|^2 + | \nabla u_1|^2)\\
	&\quad - 2\eta (1+\frac{1}{2} d \tan \frac{d}{2} ) | \nabla d|^2 -2 d \nabla \eta \cdot \nabla d + \textup{Quad} (\eta, |\nabla \eta|). 
\end{split}
\end{align*}	
\end{corollary}

\section{Energy Convexity, Existence, Uniqueness, and Subharmonicity} \label{energy-convexity} As with the previous section, the results in this section are stated in \cite{serbinowski}. Excepting the first theorem, they are stated without proof. As, again, the calculations are non-trivial and tedious, we verify them for the reader.

\begin{theorem}[ {\cite[Proposition 1.15]{serbinowski}}] \label{theorem:energy-convexity}
 Let $u_0, u_1: \Omega \to \overline{\mathcal B_\rho(O)}$ be finite energy maps with $\rho \in (0, \frac \pi 2)$. Denote by 
\begin{align*}
	d(x) &= d(u_0(x), u_1(x))\\
	R(x) & = d(u_{\frac{1}{2}}(x), O).
\end{align*}
Then there exists a continuous function $\eta(x): \Omega \to [0,1]$ such that the function $w: \Omega \to \overline{\mathcal B_{\rho}(O)}$ defined by
\[
	w(x) = (1-\eta(x)) u_{\frac{1}{2}}(x) + \eta(x) O
\]
is in $W^{1,2}(\Omega, \overline{B_{\rho}(O)})$ and satisfies
\begin{align*}
	 (\cos^8\rho) \int_{\Omega} \left| \nabla \frac{\tan \frac{1}{2} d}{\cos R} \right|^2 \, d\mu_g \le \frac{1}{2} \left(\int_\Omega | \nabla u_0|^2d\mu_g  + \int_{\Omega} | \nabla u_1|^2  d\mu_g\right)  - \int_\Omega | \nabla w|^2 d\mu_g.
\end{align*}
\end{theorem}
\begin{proof}
Once the estimates in Lemma \ref{lemma:estimateI} and Lemma~\ref{lemma:estimateII} are established, we  proceed as in~\cite{serbinowski}. Choose  $\eta $ to satisfy
\[
	\frac{\sin ((1-\eta(x))R(x))}{\sin R(x)} = \cos \frac{d(x)}{2}. 
\]	
Note that $0\le \eta \le 1$ and $\eta$ is as smooth as $d(x), R(x)$. It is straightforward to verify that $w\in L_h^2(\Omega, \overline{B_{\rho}(O)})$.

For $W \in \Gamma(\Omega)$, consider the flow $\epsilon \mapsto x(\epsilon)$ induced by $W$. 
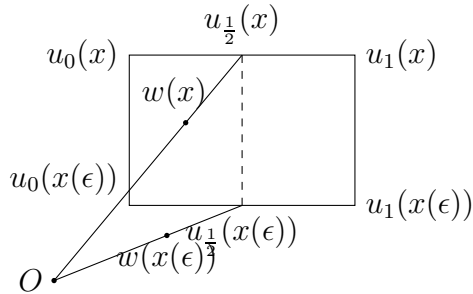
 \begin{center}
\begin{tikzpicture}
\draw (0,0) node[anchor =south east]{$u_0(x(\epsilon))$} -- (3,0) node[anchor =west] {$u_1(x(\epsilon))$} -- (3,2) node [anchor = west]{$u_1(x)$}-- (0, 2) node[anchor= east]{$u_0(x)$}-- cycle;
\draw[dashed] (1.5, 2) node[anchor = south]{$u_{\frac{1}{2}}(x)$} --(1.5,0)  node[anchor =north]  {$u_{\frac{1}{2}}(x(\epsilon))$};
\fill (-1,-1) circle (1pt) node[anchor=east]{$O$};
\draw (1.5,2) -- (-1, -1);
\draw (1.5,0) -- (-1, -1);
\fill (.75, 1.1) circle (1pt);
\node at (.6, 1.5) {$w(x)$}; 
\fill (0.5,- 0.4) circle (1pt); 
\node at (0.5, -0.7) {$w(x(\epsilon))$};
\end{tikzpicture}
\end{center}
Applying Lemma~\ref{lemma:estimateI} to the quadrilateral determined by  $P= u_0(x(\epsilon)), Q=u_0(x), R = u_1(x), S = u_1(x(\epsilon))$, divided by $\epsilon^2$, and integrate the resulting inequality against $f\in C^{\infty}_c(\Omega)$ and taking $\epsilon \to 0$, we obtain 
\[
	\left( \cos \frac{d(x)}{2}\right)^2 |(u_{\frac{1}{2}})_* (W)|^2 \le \frac{1}{2} \left( |(u_0)_*(W)|^2 + |(u_1)_*(W)|^2\right) - \frac{1}{4} | \nabla_W d|^2.
\]Note that the cubic terms vanish in the limit as every cubic term will be the product of  an $L^1$ function and $d(x)-d(x(\epsilon))$ or $d(u_i(x), u_i(x(\epsilon)))$, $i = 0, \frac 12, 1$.

 Applying Lemma~\ref{lemma:estimateII} to the triangle determined by $Q=O, P=u_{\frac{1}{2}} (x), S= u_{\frac{1}{2}}(x(\epsilon))$ yields 
 \begin{align*}
	|(w)_*(W)|^2 &\le \left(\frac{\sin(1-\eta)R}{\sin R} \right)^2 (|(u_{\frac{1}{2}})*(W)|^2 - |\nabla_W R|^2) + |\nabla_W ((1-\eta)R)|^2\\
	&= 	\left( \cos \frac{d(x)}{2}\right)^2 (|(u_{\frac{1}{2}})_*(W)|^2 - |\nabla_W R|^2) + |\nabla_W ((1-\eta)R)|^2.
\end{align*}
The above two inequalities imply
\begin{align*}
	|w_*(W)|^2 &\le \frac{1}{2} \left( |(u_0)_*(W)|^2 + |(u_1)_*(W)|^2\right) \\
	&\quad - \frac{1}{4} | \nabla_W d|^2- \left(\cos \frac{d(x)}{2}\right)^2 |\nabla_W R|^2 +  | \nabla_W \left((1-\eta) R \right)|^2.
\end{align*}
By direct computation, 
\begin{align*}
&- \frac{1}{4} | \nabla_W d|^2- \left(\cos \frac{d(x)}{2}\right)^2 |\nabla_W R|^2 +  | \nabla_W \left((1-\eta) R \right)|^2 \\
&= -\frac{\cos^4R(x) \cos^4 \frac{d(x)}{2}}{1-\sin^2 R(x) \cos^2 \frac{d(x)}{2}} \left| \nabla \frac{\tan\frac{d(x)}{2}}{\cos R(x)}\right|^2.
\end{align*}
The lemma follows from estimating 
\[
	 \frac{\cos^4R(x) \cos^4 \frac{d(x)}{2}}{1-\sin^2 R(x) \cos^2 \frac{d(x)}{2}}  \ge \cos^4R(x) \cos^4 \frac{d(x)}{2} \ge \cos^8 \rho,
\]
dividing the resulting inequality by $\epsilon^2$, integrating over $\mathbb S^{n-1}$, letting $\epsilon\to 0$, and then integrating over $\Omega$. 
\end{proof}

\begin{theorem}[Existence Theorem] \label{existence}
For  any $\rho \in (0, \frac{\pi}{4})$ and for any
finite energy map $h: \Omega \rightarrow \overline{\mathcal{B}_{\rho}(O)} \subset X$, there exists a unique element $^{Dir}h \in W^{1,2}_h(\Omega, \overline{\mathcal{B}_{\rho}(O)})$ which minimizes energy amongst all maps in $W^{1,2}_h(\Omega,\overline{\mathcal{B}_{\rho}(O)})$. 

Moreover, for any $\sigma \in (0, \rho)$,  if $^{Dir}h(\partial \Omega) \subset \overline{ \B_\sigma(O)}$ then $\overline{^{Dir}h(\Omega)} \subset \overline {\B_\sigma(O)}$.
\end{theorem}

\begin{proof}
Denote by 
\[
	E_0= \inf\{E(u):  u \in W_h^{1,2}(\Omega, \overline{\mathcal{B}_{\rho} (O)})\}.
\]
Let $u_i \in W^{1,2}(\Omega, \overline{\mathcal{B}_{\rho} (P)})$ such that $E(u_i) \to E_0$.  By Theorem~\ref{theorem:energy-convexity}, we have that
\begin{align*}
	(\cos^8\rho) \int_{\Omega} \left| \nabla \frac{\tan \frac{1}{2} d(u_k(x), u_{\ell}(x))}{\cos R} \right| \, d\mu_g \le  \frac{1}{2} \left(E(u_k)+  E(u_{\ell})\right) - E(w_{k\ell}),
\end{align*}
where $w_{k\ell}$ is the interpolation map defined by Theorem~\ref{theorem:energy-convexity}. The above right hand side goes to $0$ as $k, \ell\to \infty$. By the Poincar\'{e} inequality, 
\[
	\int_{\Omega} d(u_k, u_{\ell})\,d \mu_g \to 0.
\]  
Thus the sequence $\{ u_k\}$ is Cauchy and  $u_k \to u$ for some $u\in W^{1,2}(\Omega, \overline{\mathcal{B}_{\rho} (O)})$  because $W^{1,2}(\Omega, \overline{\mathcal{B}_{\rho} (O)})$ is a complete metric space. By trace theory, $u\in W_h^{1,2}(\Omega, \overline{\mathcal{B}_{\rho} (O)})$. By lower semi-continuity of the energy, $E(u) =E_0$. The energy minimizer is unique by energy convexity. 

Finally, since $\rho< \frac \pi 4$, for any $\sigma \in (0, \rho]$,  the ball $\B_\sigma(O)$ is geodesically convex. Therefore, the projection map $\pi_\sigma:\overline{\B_\rho(O)} \to \overline{\B_\sigma(O)}$ is well-defined and distance decreasing. Thus, since $^{Dir}h(\Omega) \subset \overline{\B_\rho(O)}$, we can prove the final statement by contradiction using the projection map to decrease energy.
\end{proof}

\begin{lemma}[cf. {\cite[(2.5)]{serbinowski}}]\label{lemma:comparison}
Let $u_0, u_1: \Omega \to \mathcal{B}_\rho(Q) \subset X$ be finite energy maps (possibly with different boundary values). For any given  $\eta \in C_c^\infty(\Omega) $ with $0\le \eta < 1/2$, there exists finite energy maps $u_\eta, \hat{u}_\eta \in W^{1,2}_{u_0} (\Omega, \B_\rho(Q))$ and  $u_{1-\eta}, \hat{u}_{1-\eta}\in W^{1,2}_{u_1} (\Omega, \mathcal B_\rho(Q))$ such that 
\begin{align*}
	&|\pi(\hat{u}_\eta)|^2 +|\pi( \hat{u}_{1-\eta})|^2 - |\pi(u_0)|^2 - |\pi(u_1)|^2 \\
	&\le - 2\cos R^{u_\eta} \cos R^{u_{1-\eta}} \nabla \left(\frac{d}{\sin d} \eta F_\eta\right) \cdot \nabla F_\eta + \textup{Quad}(\eta, \nabla \eta),
\end{align*}
where 
\begin{align*}
d(x)&= d(u_0(x), u_1(x))\\
R^{u_\eta}(x) &= d(u_\eta(x), Q)\\
R^{u_{1-\eta}}(x) &= d(u_{1-\eta}(x), Q)
\end{align*}and
\[
	F_\eta=\sqrt{\frac{1-\cos d}{\cos R^{u_{\eta}} \cos R^{u_{1-\eta}}}}.
\] 
\end{lemma}
\begin{proof}
Let $\eta \in C_c^{\infty}(\Omega)$ satisfy $0\le \eta < 1/2$.  For  $0\le \phi, \psi\le 1$ that will be determined below, we define the comparison maps
\begin{align*}
	\hat{u}_\eta &= (1-\phi(x)) u_\eta(x) + \phi(x) Q\\
	\hat{u}_{1-\eta} &= (1-\psi(x))u_{1-\eta}(x) + \psi(x) Q,
\end{align*}
where 
\[ 
u_\eta(x) = (1-\eta(x)) u_0 (x)+ \eta(x) u_1(x) \quad \mbox{and} \quad u_{1-\eta}(x) = \eta(x) u_0 (x)+(1- \eta(x)) u_1(x). 
\]
 By Corollary~\ref{corollary:interpolation-point}, 
\begin{align*}
	|\pi (\hat{u}_\eta)|^2 + | \pi( \hat{u}_{1-\eta})|^2 &\le \left(\frac{\sin(1-\phi)R^{u_\eta}}{\sin R^{u_\eta}} \right)^2 (|\pi( u_\eta) |^2 - |\nabla R^{u_\eta}|^2) + |\nabla ((1-\phi)R^{u_\eta})|^2 \\
	&\quad + \left(\frac{\sin(1-\psi)R^{u_{1-\eta}}}{\sin R^{u_{1-\eta}}} \right)^2 (|\pi( u_{1-\eta} )|^2 - |\nabla R^{u_{1-\eta}}|^2) + |\nabla ((1-\psi)R^{u_{1-\eta}})|^2.
\end{align*}
 Define $\phi$ and $\psi$ so that 
\begin{align*}
	 \frac{\sin^2 ((1-\phi)R^{u_\eta})}{\sin^2 R^{u_\eta}}  &= 1- 2\eta d \tan \frac{d}{2} + O(\eta^2)\\
	  \frac{\sin^2 ((1-\psi)R^{u_{1-\eta}})}{\sin^2 R^{u_{1-\eta}}}  &= 1- 2\eta d \tan \frac{d}{2}+ O(\eta^2).
\end{align*}
Since $\frac{\sin (1-a)\theta}{\sin \theta} = 1 - a \theta \cot\theta + O(a^2)$, we solve 
\begin{align*}
	\phi = \eta \frac{\tan R^{u_\eta}}{R^{u_\eta}} d \tan \frac{d}{2} \quad \mbox{and} \quad \psi = \eta \frac{\tan R^{u_{1-\eta}}}{R^{u_{1-\eta}}} d \tan \frac{d}{2}.
\end{align*}
Note that in particular $u_\eta, \hat{u}_\eta\in W^{1,2}_{u_0}(\Omega, \mathcal B_\rho(Q))$ and $u_{1-\eta}, \hat{u}_{1-\eta} \in W^{1,2}_{u_1}(\Omega, \mathcal B_\rho(Q))$. 

Together with the estimate for $|\pi(u_\eta)|^2 + | \pi(u_{1-\eta})|^2$ in Corollary~\ref{corollary:interpolation} (which also explains the choice of $\phi$ and $\psi$ in order to eliminate the coefficient), we have
\begin{align*}
	&|\pi( \hat{u}_\eta)|^2 + | \pi(\hat{u}_{1-\eta})|^2-  |\pi(u_0)|^2 - | \pi( u_1)|^2\\
	&\le -  2\eta (1+\frac{1}{2} d \tan \frac{d}{2} )   | \nabla d|^2 - 2 d \nabla \eta \cdot \nabla d  - (1- 2\eta d \tan \frac{d}{2}) ( | \nabla R^{u_\eta}|^2 + |\nabla R^{u_{1-\eta}}|^2)\\
	&\quad  + |\nabla (1- \eta \frac{\tan R^{u_\eta}}{R^{u_\eta}} d \tan \frac{d}{2}) R^{u_\eta}|^2 + |\nabla (1-\eta \frac{\tan R^{u_{1-\eta}}}{R^{u_{1-\eta}}} d \tan \frac{d}{2}) R^{u_{1-\eta}}|^2  + \textup{Quad} (\eta, |\nabla \eta|).
\end{align*}
Simplifying the expression and using  $1-\sec^2 \theta =-\tan^2 \theta$ , we obtain
\begin{align} \label{equation:comparison}
\begin{split}
	&\frac{1}{2} \left(|\pi( \hat{u}_\eta)|^2 + | \pi(\hat{u}_{1-\eta})|^2-  |\pi(u_0)|^2 - | \pi( u_1)|^2\right) \\
	&\quad \le \eta\bigg( -  (1+\frac{1}{2} d \tan \frac{d}{2} ) | \nabla d|^2 - d \tan \frac{d}{2} ( \tan^2 R^{u_\eta} | \nabla R^{u_\eta}|^2  + \tan^2 R^{u_{1-\eta}}|\nabla R^{u_{1-\eta}}|^2)\\
	&\quad\quad   - \nabla  (d \tan \frac{d}{2} ) \cdot (\tan R^{u_\eta} \nabla R^{u_\eta}+\tan R^{u_{1-\eta}}   \nabla R^{u_{1-\eta}})\bigg)\\
	&\quad \quad  + \nabla \eta \cdot \left( -d \nabla d -\tan R^{u_\eta} d \tan \frac{d}{2} \nabla R^{u_\eta} -\tan R^{u_{1-\eta}} d \tan \frac{d}{2} \nabla R^{u_{1-\eta}} \right) + \textup{Quad}(\eta, \nabla \eta).
	\end{split}
\end{align}
We  hope to find $a, b, F_\eta$ which are functions of $d, R^{u_\eta}$ and  $R^{u_{1-\eta}} $ such that the right hand side above is   $\le a \nabla (b\eta F_\eta) \cdot \nabla F_\eta$. 

Since $a \nabla (b\eta F_\eta) \cdot \nabla F_\eta=  \eta ( ab |\nabla F_\eta|^2 + \frac{a}{2} \nabla b\cdot \nabla F_\eta^2 )+ \frac{ab}{2} \nabla \eta \cdot \nabla F_\eta^2$, by comparing the terms involving $\nabla \eta$ in \eqref{equation:comparison}, we  solve 
\begin{align*}
	 \frac{ab}{2} \nabla \eta \cdot \nabla F_\eta^2&= \nabla \eta \cdot \left( -d \nabla d -\tan R^{u_\eta} d \tan \frac{d}{2} \nabla R^{u_\eta} -\tan R^{u_{1-\eta}} d \tan \frac{d}{2} \nabla R^{u_{1-\eta}} \right)\\
	&=- d \tan \frac{d}{2}  \nabla \eta  \cdot \left(  \nabla \log \sin^2 \frac{d}{2} -\nabla \log \cos R^{u_\eta} - \nabla \log  \cos R^{u_{1-\eta}}  \right)  \\
	&=- \frac{d}{\sin d} \cos R^{u_{\eta}} \cos R^{u_{1-\eta}} \nabla  \eta \cdot \nabla \frac{1-\cos d}{\cos R^{u_{\eta}} \cos R^{u_{1-\eta}}},
\end{align*}
where we use  $2\sin^2 \frac{d}{2} = (1-\cos d)$ and $\tan \frac{d}{2} = \frac{1-\cos d}{\sin d}$. It suggests us to choose 
\[
	\frac{ab}{2} =- \frac{d}{\sin d} \cos R^{u_{\eta}} \cos R^{u_{1-\eta}} \quad \mbox{and} \quad   F_\eta=\sqrt{\frac{1-\cos d}{\cos R^{u_{\eta}} \cos R^{u_{1-\eta}}}}.
\]
We then compute the term $ \eta ( ab |\nabla F_\eta|^2 + \frac{a}{2} \nabla b\cdot \nabla F_\eta^2 )$ for the above choices of $a, b,$ and  $F_\eta$. For the term $ab |\nabla F_\eta|^2$, we compute
\begin{align*}
	ab |\nabla F_\eta|^2 &= -\frac{d}{2 \sin d(1-\cos d)}  \left| \sin d \nabla d+ (1-\cos d)(\tan R^{u_\eta} \nabla R^{u_\eta} + \tan R^{u_{1-\eta}} \nabla R^{u_{1-\eta}}) \right|^2\\
	&\ge -\left( \frac{d\sin d}{2(1-\cos d)} |\nabla d|^2 + d \nabla d \cdot (\tan R^{u_\eta} \nabla R^{u_\eta} + \tan R^{u_{1-\eta}} \nabla R^{u_{1-\eta}})\right.\\
	&  \quad \quad \left. +\frac{d(1-\cos d)}{\sin d}(\tan^2 R^{u_\eta} |\nabla R^{u_\eta} |^2+ \tan^2 R^{u_{1-\eta}} |\nabla R^{u_{1-\eta}}|^2)\right),
\end{align*}
where we expand the quadratic term and use the AM-GM inequality to handle the cross term  $(\tan R^{u_\eta} \nabla R^{u_\eta}) \cdot  (\tan R^{u_{1-\eta}}\nabla R^{u_{1-\eta}})$.  For the term $\frac{a}{2} \nabla b\cdot \nabla F_\eta^2$, we assume $b=b(d)$ and compute:
\begin{align*}
	\frac{a}{2} \nabla b\cdot \nabla F_\eta^2  &= \frac{ab}{2} \nabla \log b\cdot \nabla F_\eta^2 \\
	&= -d \frac{b'}{b}  |\nabla d |^2-\frac{d(1-\cos d)}{\sin d} \frac{b'}{b}  \nabla d\cdot( \tan R^{u_\eta} \nabla R^{u_\eta} +  \tan R^{u_{1-\eta}} \nabla R^{u_{1-\eta}}  ).
\end{align*}
Combining the above inequalities, we obtain
\begin{align*}
	ab |\nabla F_\eta|^2+\frac{a}{2} \nabla b\cdot \nabla F_\eta^2& \ge - \left[  \left( \frac{d\sin d}{2(1-\cos d)} + d\frac{b'}{b} \right) |\nabla d|^2 \right.\\
	&\quad \quad \left.+ \left(d + \frac{d(1-\cos d)}{\sin d} \frac{b'}{b}  \right) \nabla d\cdot( \tan R^{u_\eta} \nabla R^{u_\eta} +  \tan R^{u_{1-\eta}} \nabla R^{u_{1-\eta}} )\right.\\
	&\quad \quad \left. +\frac{d(1-\cos d)}{\sin d}(\tan^2 R^{u_\eta} |\nabla R^{u_\eta} |^2+ \tan^2 R^{u_{1-\eta}} |\nabla R^{u_{1-\eta}}|^2) \right].
\end{align*}
Comparing to \eqref{equation:comparison}, we solve
\begin{align*}
	 \frac{d\sin d}{2(1-\cos d)}\nabla d + d \nabla \log b  &= (1+ \frac{1}{2} d \tan \frac{d}{2} )\nabla d\\
	 d\nabla d+ \frac{d(1-\cos d)}{\sin d} \nabla \log b   & = \nabla (d\tan \frac{d}{2}).
\end{align*}
which implies that $b = \frac{d}{\sin d}$, and hence $a = -2 \cos R^{u_{\eta}} \cos R^{u_{1-\eta}}$.

\end{proof}

\begin{theorem}[cf. {\cite[Corollary 2.3]{serbinowski}}] \label{theorem:subharmonicity}
Let  $u_0, u_1: \Omega \to \mathcal{B}_\rho(P) \subset X$ be a pair of energy minimizing maps (possibly with different boundary values). Let $d(x) = d(u_0(x), u_1(x))$ and $R^{u_i} = d(u_i, P)$. Then the function
\[
	F= \sqrt{\frac{1-\cos d}{\cos R^{u_0} \cos R^{u_1}}}
\]
satisfies the differential inequality weakly
\[
	\textup{div} (\cos R^{u_0} \cos R^{u_1} \nabla F)\ge 0.
\] 
\end{theorem}

\begin{proof}
Let $\eta \in C^\infty_c(\Omega)$ with $\eta \ge 0$. For $t>0$ sufficiently small, we have $0\le t\eta < 1/2$. Let $\hat{u}_{t\eta}$ and $\hat{u}_{1-t\eta}$ be the corresponding maps defined as  in Lemma~\ref{lemma:comparison}. Since $u_0$ and $u_1$ minimize the energy among maps of the same boundary values, we  have
 \begin{align*}
 	&0\le\int_\Omega |\pi(\hat{u}_\eta)|^2 +|\pi( \hat{u}_{1-\eta})|^2 - |\pi(u_0)|^2 - |\pi(u_1)|^2 \, d\mu_g\\
	& \le \int_\Omega - 2\cos R^{u_{t\eta}} \cos R^{u_{1-{t\eta}}} \nabla \left(\frac{d}{\sin d} t\eta F_{t\eta}\right) \cdot \nabla F_{t\eta} \, d\mu_g + t^2\textup{Quad}(\eta, \nabla \eta). 
 \end{align*}
 Dividing the inequality by $t$ and let $t\to 0$, since $R^{u_{t\eta}}\to R^{u_0}$ and $R^{u_{1-{t\eta}}}\to R^{u_1}$ and $F_{t\eta} \to F$,  we derive
\begin{align*}
 0&\le \int_\Omega - 2\cos R^{u_0} \cos R^{u_1} \nabla \left(\frac{d}{\sin d} \eta F\right) \cdot \nabla F\, d\mu_g \\
 &=2 \int_\Omega  \left(\frac{d}{\sin d} \eta F\right) \textup{div} \left(\cos R^{u_0} \cos R^{u_1}\nabla F\right)\, d\mu_g.
 \end{align*}
\end{proof}


\end{document}